\documentclass[oneside,english]{amsart}

\usepackage[latin1]{inputenc}
\usepackage[dvipsnames]{xcolor}
\usepackage{setspace} 
\usepackage{amssymb}
\usepackage{stmaryrd}
\usepackage[]{epsf,epsfig,amsmath,amssymb,amsfonts,latexsym}
\usepackage{enumerate}
\usepackage{mathrsfs}
\usepackage{wasysym}
\usepackage[inline]{enumitem}
\usepackage{mathtools}
\usepackage{braket}		
\usepackage[normalem]{ulem}		
\usepackage{nicefrac}
\usepackage{cancel}
\usepackage{xcolor}
\usepackage[a4paper, left=2.7cm, right=2.7cm, top=3.5cm, bottom=5cm]{geometry}
\usepackage[pdftex=true,hyperindex=true,pdfborder={0 0 0}]{hyperref}
\usepackage{doi}
\usepackage[numbers]{natbib}
\usepackage{cleveref}
\crefname{subsection}{subsection}{subsections}
\usepackage{tikz}
\usetikzlibrary{shapes,arrows,automata,matrix,fit,patterns}
\usepackage{soul} 

\usepackage{tikz-cd}
\usepackage{autonum}		

\theoremstyle{plain}
\newtheorem{theorem}{Theorem}[section]
\newtheorem{lemma}[theorem]{Lemma}
\newtheorem{corollary}[theorem]{Corollary}
\newtheorem{proposition}[theorem]{Proposition}

\newtheorem{question}[theorem]{Question}
\theoremstyle{definition}
\newtheorem{definition}[theorem]{Definition}
\newtheorem{remark}[theorem]{Remark}
\newtheorem{example}[theorem]{Example}

\newcommand{\ZZ}{\mathbb{Z}}			
\newcommand{\NN}{\mathbb{N}}			
\newcommand{\RR}{\mathbb{R}}			
\newcommand{\CC}{\mathbb{C}}            
\newcommand{\symb}[1]{\mathtt{#1}}		
\newcommand{\isdef}{\coloneqq}			
\makeatletter
\newcommand{\vast}{\bBigg@{4}}
\newcommand{\Vast}{\bBigg@{5}}

\makeatother
	
\newcommand{\supp}{
	\operatorname{\mathrm{supp}}%
}

\newcommand{\ii}{\mathrm{i}}				

		\global\long\def\Z{\mathbb{Z}}%
		\global\long\def\N{\mathbb{N}}%
		\global\long\def\F{\mathcal{F}}%

\newcommand{\define}[1]{\textbf{#1}}

\newcommand{\xConfig}[1]{%
	\begin{tikzpicture}[
		baseline=-\the\dimexpr\fontdimen22\textfont2\relax,ampersand replacement=\&]
		\matrix[
		matrix of math nodes,
		nodes={
			minimum size=1.4ex,text width=1.4ex,
			text height=1.4ex,inner sep=3pt,draw={gray!20},anchor=center
		}, row sep=1pt,column sep=1pt
		] (config) {#1};
		\node[draw,rectangle,help lines,gray!50, dashed,fit=(config),inner sep=-1pt] {};
	\end{tikzpicture}
}

\newtheorem{maintheorem}{Theorem}

\title{%
	Effective dynamical systems beyond dimension zero and factors of SFTs
}

\author{Sebasti\'an Barbieri, Nicanor Carrasco-Vargas and Crist\'obal Rojas}

\newcommand{\Addresses}{{
		\bigskip
		
		\hskip-\parindent   S.~Barbieri, \textsc{Departamento de Matem\'{a}tica y ciencia de la computaci\'{o}n, Universidad de Santiago de Chile, Santiago, Chile.}\par\nopagebreak
		\textit{E-mail address}: \texttt{sebastian.barbieri@usach.cl}
		
		\medskip
		
		\hskip-\parindent   N.~Carrasco-Vargas, \textsc{Departamento de Matem\'atica, Pontificia Universidad Cat\'olica de Chile, Santiago, Chile.}\par\nopagebreak
		\textit{E-mail address}: \texttt{njcarrasco@mat.uc.cl}
		
		\medskip
		
		\hskip-\parindent   C.~Rojas, \textsc{Instituto de Ingenier\'ia Matem\'atica y Computacional, Pontificia Universidad Cat\'olica de Chile, Santiago, Chile.}\par\nopagebreak
		\textit{E-mail address}: \texttt{cristobal.rojas@mat.uc.cl}
}}
\date{}

\begin{document}
	
	\maketitle
 
	\begin{abstract}

 Using tools from computable analysis we develop a notion of effectiveness for general dynamical systems as those group actions on arbitrary spaces that contain a computable representative in their topological conjugacy class. Most natural systems one can think of are effective in this sense, including some group rotations, affine actions on the torus and finitely presented algebraic actions. We show that for finitely generated and recursively presented groups, every effective dynamical system is the topological factor of a computable action on an effectively closed subset of the Cantor space. We then apply this result to extend the simulation results available in the literature beyond zero-dimensional spaces. In particular, we show that for a large class of groups, many of these natural actions are topological factors of subshifts of finite type.  
		\medskip
		
		
	\end{abstract}
	

\section{Introduction}

 Starting with the work of Hadamard~\cite{Hadamard1898} and the highly influential article of Morse and Hedlund~\cite{MorseHedlund1938}, symbolic dynamical systems have quite often played a pivotal role in the understanding of more general dynamics. A celebrated instance of this is the prominent role of subshifts of finite type (SFT) in the study of Anosov, and more generally of Axiom A diffeomorphisms, through Markov partitions of their non-wandering sets~\cite{Bowen1978}. Another well-known example of this tight relationship is the fact that the natural action of a word-hyperbolic group on its boundary is a very well behaved topological factor of an SFT on the same group~\cite{Coornaert2006-du}.

 These and similar results raise the question of understanding precisely which dynamical systems are topological factors of SFTs. An observation made by Hochman~\cite{Hochman2009b} is that subactions of multidimensional SFTs satisfy strong computability constraints. More precisely, they must be computable maps on effectively closed sets in the sense that the complement of the space of orbits must consist on a union of cylinders whose defining words can be enumerated by a Turing machine. However, the truly remarkable discovery of Hochman is that, up to a difference in the dimension of the acting group and a topological factor, this is the only constraint: every computable homeomorphism on an effectively closed zero-dimensional set is the topological factor of a subaction of a $\ZZ^3$-SFT.

Results linking computable maps on effectively closed zero-dimensional sets to SFTs are called  ``simulation results'', as they express that very explicit and simple models such as SFTs are capable of universally encoding this considerably larger class of dynamical systems. These simulation results along with further developments~\cite{AubrunSablik2010,DurandRomashchenkoShen2010} led to a new understanding of classical results in the theory of symbolic dynamics of group actions, such as the undecidability of the domino problem~\cite{Wang1961,Berger1966}, the existence of aperiodic tilesets~\cite{Berger1966,Robinson1971} and more generally the existence of two-dimensional tilings without computable orbits~\cite{Hanf1974,Myers1974}. Moreover, they also provided the tools to obtain new and long sought-after results, such as the classification of topological entropies of multidimensional SFTs~\cite{HochmanMeyerovitch2010}.

Further work has extended the initial result of Hochman to the context of actions of discrete groups on zero-dimensional spaces~\cite{BS2018,Barbieri_2019_DA,barbieri2023soficity}. Notably, it was shown that for a large class of non-amenable groups called self-simulable, the class of zero-dimensional topological factors of SFTs contains every possible computable action on an effectively closed zero-dimensional set~\cite{Barbieri_Sablik_Salo_2021}. These works have led to new results about the dynamics of such groups. For instance, they have provided new examples of groups that can act freely, expansively and with shadowing on a zero-dimensional space.

A common theme among all of the previous simulation results is that they apply to actions on zero-dimensional spaces. The reason behind this is rooted in the fact that the Cantor space $\{0,1\}^{\NN}$ admits a natural computable structure, where the cylinders are described by finite words, and that allows the application of algorithmic techniques. The main objective of this paper is to explore a generalization to groups acting on non-symbolic spaces, such as compact subsets of $\RR^n$, $\operatorname{GL}_n(\CC)$ or compact abelian groups such as $(\RR/\ZZ)^n$. Thus, the goal of this article is to explore the following question.

\begin{question}
    Can the simulation results be extended to group actions on spaces that are not zero-dimensional? 
\end{question}
    
The theory of computable analysis offers the means to endow separable metric spaces with computable structures that allow the notion of computable map to make sense, and the application of algorithmic techniques possible. In this paper we explore this approach and study its connections with the existing simulation results. In particular, we introduce a very general notion of effective dynamical system (EDS) and show that all of the known simulation results can be extended to this class. By an effective dynamical system we mean one that is topologically conjugate to a computable action over a recursively compact subset of a computable metric space. We do not require the topological conjugacy to be computable, so an EDS does not need to be computable itself. The class of EDS is therefore quite large and encompasses virtually all natural examples (although artificial non-examples can be constructed, see \Cref{sec:factors}). Our main tool will be the following result: 

\begin{maintheorem}\label{thm:zero_dim_effective_extension}
				Let $\Gamma$ be a finitely generated and recursively presented group. For any effective dynamical system $\Gamma \curvearrowright X$ there exists an effectively closed zero-dimensional space $\widetilde{X}\subset \{0,1\}^{\NN}$ and a computable action $\Gamma \curvearrowright \widetilde{X}$ such that $\Gamma \curvearrowright X$ is a topological factor of $\Gamma \curvearrowright \widetilde{X}$.
\end{maintheorem}


We remark that is it a well-known fact that every group action by homeomorphisms on a compact metrizable space admits a zero-dimensional extension. This extension, however, carries a priory no computable structure. Our contribution is that, for recursively presented groups, the computable structure stemming from the effective nature of the system is preserved by a well-chosen zero-dimensional extension.

Given a finitely generated group $\Gamma$ and an epimorphism $\psi \colon \Gamma \to H$, we say that $\Gamma$ simulates $H$ if given any computable action of $H$ on an effectively closed subset of the Cantor space, the corresponding action of $\Gamma$ induced through $\psi$, is a factor of a $\Gamma$-SFT. If $\psi$ is an isomorphism, then $\Gamma$ is called self-simulable. The class of self-simulable groups includes several interesting examples, see \Cref{thm:simulation_results}.

The main application of~\Cref{thm:zero_dim_effective_extension} is that for recursively presented groups the simulation results in the literature also apply to EDS.

\begin{maintheorem}\label{thm:simulation-by-SFTs-enhanced-version-GOTY}
        Let $\Gamma,H$ be finitely generated groups and $\psi\colon \Gamma \to H$ be an epimorphism. Suppose that $H$ is recursively presented, then $\Gamma$ simulates $H$ if and only if for every effective dynamical system $H \curvearrowright X$ the induced action of $\Gamma$ is a topological factor of a $\Gamma$-SFT.
\end{maintheorem}

Combining this result with the simulations result in the literature, we obtain that many dynamical systems on non zero-dimensional spaces can be realized as factors of subshifts of finite type. This includes the natural action of $\operatorname{GL}_n(\ZZ) \curvearrowright \RR^n/\ZZ^n$ by left matrix multiplication for $n \geq 5$, actions of $F_2\times F_2$ on the circle by computable rotations on each generator, actions of the braid groups $B_n$ for $n \geq 7$ on compact computable groups, and more. 


To better illustrate the power of \Cref{thm:simulation-by-SFTs-enhanced-version-GOTY}, we develop in detail a particular application which we believe could be special interest. An action of a countable group on a compact abelian topological group by automorphisms is called algebraic. This is a rich class of group actions which has been the object of much study in the literature, for instance in~\cite{Lind1990,Schmidt1995,LindSchmidt99,Chung2014}. Using our results we obtain that a large class of algebraic actions of non-amenable groups can be presented as topological factors of SFTs.

\begin{maintheorem}\label{thm:finitely_presented_algebraic_action}
    Let $\Gamma$ be a finitely generated and recursively presented self-simulable group. Every finitely presented algebraic action of $\Gamma$ is the topological factor of a $\Gamma$-subshift of finite type.
\end{maintheorem}

In fact, this result applies to an even larger class of algebraic actions which admit ``recursive presentations'', see~\Cref{def:recpresented_algebraic_action}.

We complete our discussion with several observations. There exist at least two notions of computability for shift spaces present in the literature (\Cref{cor:ECP_plus_RP_implies_EDS}). We discuss their connection to the notion of EDS and show that for recursively presented groups they all coincide. We also completely characterize the relations between these notions for groups which are not recursively presented. We then focus on the class of topological factors of EDS. While it is true that computable factors of EDS are EDS (\Cref{prop:computable_factor_EDS_is_EDS}), the class of EDS is in general not closed under topological factor maps. We provide two examples of this fact, one as a non-EDS factor of an EDS on a compact subset of $\CC$, and a non-EDS zero-dimensional factor of a zero-dimensional EDS. In the case of zero-dimensional spaces, we propose a new notion (\Cref{def:WEDS}) of weak effective dynamical system (WEDS). We show that for recursively presented groups, every zero-dimensional factor of an EDS is a WEDS (\Cref{prop:factor-of-EDS-is-WEDS}), and that the class of WEDS is closed under topological factor maps. This class of systems can be though of as those that can be written as an inverse limit of effective subshifts, but with the caveat that the effectiveness of the sequence of subshifts is not required to be uniform. Using this notion, we naturally recover the fact that the class of subshifts which are EDS is closed under topological factor maps.

\subsection{Paper organization} After some preliminaries, we give an introduction to computable analysis and the precise definition of an EDS in~\Cref{sec:comp}. The proof of~\Cref{thm:zero_dim_effective_extension} along with a few interesting side remarks are provided in~\Cref{sec:resultado}. In particular, we show that for zero-dimensional EDS we may always consider the canonical computable structure induced by the cylinder sets and that in this case we do not need the hypothesis that the acting group is recursively presented. We also show that in case the action admits a generating cover, the extension can be taken to be an effective subshift. In~\Cref{sec:simu} we summarize the simulation results in the literature, prove  ~\Cref{thm:simulation-by-SFTs-enhanced-version-GOTY}, and discuss several examples. In~\Cref{sec:algebraic actions} we present the relevant background for ~\Cref{thm:finitely_presented_algebraic_action} along with its proof. In~\Cref{sec:computability-on-shift-spaces} we turn our attention back to shift spaces and study the relationship between the different notions of computability for subshifts. Finally, in~\Cref{sec:factors} we construct the examples showing that the class of EDS is not closed by topological factor maps, introduce the class of WEDS, show that this class is closed by factor maps, and discuss their relationship with EDS.  
 \medskip
 
	 \textbf{Acknowledgments}: The authors are thankful to Mathieu Hoyrup for stimulating conversations on the construction of explicit examples of non effective dynamical systems in higher dimensions. We also thank an anonymous reviewer for suggesting several substantial improvements.
  S. Barbieri was supported by the FONDECYT grants 11200037 and 1240085. N. Carrasco-Vargas  was partially supported by ANID CONICYT-PFCHA/Doctorado Nacional/2020-21201185, ANID/Basal National Center for Artificial Intelligence CENIA FB210017, and the European Union's Horizon 2020 research and innovation program under the Marie Sklodowska-Curie grant agreement No 731143. C. Rojas was supported by ANID/FONDECYT Regular 1230469 and ANID/Basal National Center for Artificial Intelligence CENIA FB210017.
    
\section{Preliminaries}\label{sec:prelim}
\subsection{Computability on countable sets}
A formal introduction to the theory of computation can be found in~\cite{Rogers:1987:TRF:28907} or~\cite{sipser2012introduction}. Here we will provide a brief functional description. We will make use of the word \define{algorithm} to refer to a Turing Machine~\cite{Turing1936}. The reader may think about a Turing machine as a computer program written in any standard programming language which receives some finite information called \define{input}, then proceeds to sequentially execute a finite set of instructions. Depending on these instructions and the input, the execution may at some halt and return some finite information, called \define{output}, or it may not halt, in which case the execution will continue forever without ever providing an output. 

The notion of algorithm characterizes mathematical objects according to the extent to which they can be \emph{computed}. Depending on what exactly the algorithm is required to compute, we will use different names for them in our definitions (computable, effective, decidable, etc...). 

Let us start by functions over words.  For a set $A$, denote by $A^* = \bigcup_{n \in \NN}A^n$ the set of all words on $A$. Let $A,B$ be finite sets (called \define{alphabets} in this setting). A \define{partial} function $f\colon A^* \to B^*$ is a function which is not necessarily defined in all of $A^*$. A \define{computable function} is a partial function $f\colon A^* \to B^*$ for which there exists an algorithm which on input $w \in A^*$ halts if and only if $f$ is defined on $w$, in which case it outputs $f(w)$. We say that a partial function is \define{total} if it is defined on every element of $A^*$.

We may also define computable functions on other countable sets by representing them through words in a canonical way. For instance, we will represent non-negative integers using the alphabet $A = \{\symb{0},\symb{1}\}$ through their binary representation. Note that this is a ``good" representation in the sense that comparing or adding integers, boils down to computable functions on their representations. Similarly, we can represent rational numbers, tuples in $\NN^d$ or sequences in $\NN^*$ by words using some bijection with binary words. A \define{named set} is then a pair $(X,\nu)$ where $\nu\colon\{\symb{0},\symb{1}\}^* \to X$ is a bijection that we will call \define{representation}. Again, we will choose representations which are ``good" in the sense that they make the relevant structure of the named set computable through their representations. In what follows, we will therefore speak freely of algorithms working on named sets without explicitly referring to their representations.  

Let $X$ be a named set and $A \subset X$. We say that $A$ is \define{recursively enumerable (r.e.)} or \define{semi-decidable} if there is an algorithm which, on input $w$, halts if and only if $w\in A$. Equivalently, $A$ is r.e. if it is either empty, or it equals $\{f(n) : n\in\N\}$ for some total computable function $f\colon \NN \to X$. We say that $A$ is \define{decidable} if both $A$ and $X\setminus A$ are semi-decidable.

\begin{example}
Let $(\varphi_e)_{e\in\NN}$ be an effective enumeration of all partial computable functions (for example, in lexicographic order according to their ``code''). The famous \define{halting set} $\texttt{Halt}=\{ e \in \NN : \varphi_e(e) \mbox{ halts}\}$ is an example of a semi-decidable set that is not decidable.
\end{example}

Every time we have a sequence $(x_n)_{n\in\N}$ of objects that are computable in some way, for instance a sequence $(f_n)_{n\in\N}$ of computable functions or a sequence $(A_n)_{n\in\N}$ of decidable sets, we will say that the computability of the sequence is \define{uniform} when all the objects in the sequence can be computed by the same algorithm. For example, a sequence $(f_n)_{n\in\N}$ of total functions is \define{uniformly computable} if there is an algorithm which on input $(n,w)$ halts and outputs $f_n(w)$. Similarly, a sequence $(A_n)_{n\in\N}$ of sets is \define{uniformly r.e.} if there is an algorithm which on input $(n,w)$ halts if and only if $w\in A_n$. 

\begin{example}
Let $(\varphi_e)_{e\in \NN}$ be the lexicographic enumeration of all partial computable functions and let $A_e = \{n\in \NN : \varphi_e(n) \text{ halts}\}$. Then the sequence of all recursively enumerable sets $(A_e)_{e\in\NN}$ is uniform. On the other hand, the sequence $(A_{e_n})_{n\in\NN}$ of all decidable sets is not uniform. 
\end{example}

\subsection{Finitely generated groups}

Let $\Gamma$ be a group. Given $S\subset \Gamma$ and a word $w = w_1\dots w_n \in S^*$, we denote by \underline{$w$} the element of $\Gamma$ obtained by multiplying the $w_i$. Recall that $\Gamma$ is \define{finitely generated} if $\Gamma = \{\underline{w} : w \in S^*\}$ for some finite $S\subset \Gamma$. We call such a set $S$, a \define{generator} for $\Gamma$. In this paper we will always assume generators to be symmetric, i.e., closed by inverses. The \define{word problem} of $\Gamma$ with respect to $S$ is the set of words \[\texttt{WP}_S(\Gamma)= \{w \in S^* : \underline{w} = 1_{\Gamma}\}.\]

We say that a finitely generated group $\Gamma$ is \define{recursively presented} if $\texttt{WP}_S(\Gamma)$ is recursively enumerable, and that $\Gamma$ has \define{decidable word problem} if $\texttt{WP}_S(\Gamma)$ is decidable. These two properties do not depend upon the choice of $S$, as long as $S$ is a finite generating set for $\Gamma$. 
A surjective group homomorphism is called an \textbf{epimorphism}.

\begin{remark}\label{rem:decidablewp_isnamed}
    Let $\Gamma$ be a group which is finitely generated by $S$. Consider the canonical map $\pi\colon S^* \to \Gamma$ which assigns to a word $w$ its corresponding element $\pi(w)=\underline{w}$. If the word problem of $\Gamma$ is decidable, then one can compute a set of words that uniquely represent the elements of $\Gamma$. By identifying this set with $\{\symb{0},\symb{1}\}^*$ through a computable bijection, one obtains a representation that makes the group operation computable. Conversely, it is not hard to see that if a finitely generated group admits a representation that makes the group operation computable, then it necessarily has decidable word problem. 
\end{remark}

\subsection{Group actions and their dynamics}

Given a group $\Gamma$ and actions by homeomorphisms on compact metrizable spaces $\Gamma \curvearrowright X$ and $\Gamma \curvearrowright Y$, a continuous surjective map $f\colon X \to Y$ is a \define{topological factor map} if for every $g\in \Gamma$ and $x \in X$ we have $g(f(x)) = f(gx)$. Furthermore, we say that a topological factor map is a \define{topological conjugacy} if it is an homeomorphism. A major challenge in dynamical systems theory is to classify this kind of group actions up to topological conjugacy or topological factor maps.

Given a group homomorphism $\psi\colon G\to\Gamma$ and a group action $\Gamma\curvearrowright X$, we define the \textbf{induced} action $G\curvearrowright X$ by $g x=\psi(g) x$ for every $g\in G$. 

\subsection{Shift spaces}\label{subsec:shift_spaces}

Let $A$ be an alphabet set and $\Gamma$ be a group. The \define{full $\Gamma$-shift} is the set $A^{\Gamma} = \{ x\colon \Gamma \to A\}$ equipped with the left \define{shift} action $\Gamma \curvearrowright A^{\Gamma}$ by left multiplication given by 
\[ gx(h) \isdef x(g^{-1}h) \qquad \mbox{  for every } g,h \in \Gamma 
\mbox{ and } x \in A^\Gamma. \]
  The elements $x \in A^\Gamma$ are called \define{configurations}. Given a finite set $F\subset \Gamma$, a \define{pattern} with support $F$ is an element $p \in A^F$. We denote the cylinder generated by $p$ by $[p] = \{ x \in A^{\Gamma} : x|_F = p \}$ and note that the cylinders are a clopen base for the prodiscrete topology on $A^{\Gamma}$. 
 
\begin{definition}
	A subset $X \subset A^\Gamma$ is a \define{$\Gamma$-subshift} if and only if it is $\Gamma$-invariant and closed in the prodiscrete topology. 
\end{definition}

If the context is clear, we drop the $\Gamma$ from the notation and speak plainly of a subshift.  Equivalently, $X$ is a subshift if and only if there exists a set of forbidden patterns $\F$ such that \[X=X_\F = \{ x \in A^{\Gamma} : gx\notin [p] \mbox{ for every } g \in \Gamma, p \in \F  \}.\]

A subshift $X$ is of \define{finite type} (SFT) if there exists a finite set of forbidden patterns $\mathcal{F}$ for which $X = X_{\mathcal{F}}$.

We say that an action $\Gamma \curvearrowright X$ on a compact metrizable space is \define{expansive} if for some metric $d$, which is compatible with the topology, there exists $C>0$ such that whenever $x,y \in X$ are distinct then $\sup_{g \in \Gamma} d(gx,gy) \geq C$. It is a well known elementary fact (for instance, see~\cite[Theorem 2.1]{Hedlund1969}) that an action $\Gamma \curvearrowright X$ on a zero-dimensional compact metrizable space is expansive if and only if it is topologically conjugate to a subshift.



\subsection{Computability on Cantor spaces}

Let $A$ be a finite set with $|A|\geq 2$. The space $A^{\NN}$ endowed with the product of the discrete topology will be called \define{Cantor space}. It is compact, metrizable and has topological dimension zero. The Cantor space is universal in the sense that every compact, metrizable space with topological dimension zero is homeomorphic to a closed subset of $A^{\NN}$ (Further on in~\Cref{thm:compacto-y-cero-dimensional-tiene-copia-en-el-cantor} we will prove an effective version of this result). As we will only deal with compact metrizable spaces, in what follows we will just say that a space is \define{zero-dimensional} to say that it has these three properties.

A natural metric in the Cantor space is given by \[ d(x,y) = 2^{-\inf \{n \in \NN\ :\ x_n \neq y_n \}  } \mbox{ for } x,y \in A^{\NN}.\]  

Given a word $w=w_0\dots w_n \in A^*$, the \define{cylinder} determined by $w$ is the clopen set \[ [w] = \{x \in A^{\NN} : x_0 \dots x_n = w_0 \dots w_n\}. \]
The collection of all cylinders forms a basis for the topology on $A^{\NN}$.

Next we will introduce computability notions for $A^{\NN}$. We say that a set $U \subset A^{\NN}$ is \define{effectively open} if there exists a recursively enumerable set $L\subset A^{*}$ such that $U = \bigcup_{w \in L}[w]$. Intuitively, this means that this set can be approximated by an algorithm from the inside. Similarly, we say that a set $C \subset A^{\NN}$ is \define{effectively closed}, if it is the complement of an effectively open set. For a subspace $X\subset A^{\NN}$, we say that a set is effectively open (respectively closed) in $X$ if it is the intersection of $X$ with an effectively open (respectively closed) set.

\begin{remark}
    Let $(U_i)_{i \in \NN}$ be an uniform sequence of effectively open sets and write $U_i = \bigcup_{w \in L_i}[w]$. As the union of a uniform sequence (in particular a finite sequence) of recursively enumerable languages is recursively enumerable, it follows that the union $\bigcup_{i\in \NN} U_i$ is also effectively open. An analogous argument shows that the intersection of any uniform sequence (in particular, a finite sequence) of effectively closed sets is effectively closed.
\end{remark}

\begin{remark}\label{rem: Cantor is rec.compact}
    Given a finite list of words $w_1,\dots,w_n \in A^*$, it is easy to write an algorithm which decides if $\bigcup_{i =1}^n [w_i] = A^{\NN}$. In particular, as $A^{\NN}$ is compact it follows that if $U = \bigcup_{w \in L}[w]$ is an effectively open set, we may semi-decide from $L$ whether $U = A^{\NN}$ by repeatedly checking whether finite unions of its associated cylinder sets cover $A^{\NN}$. Moreover, taking complements, we deduce that we can semi-decide if an effectively closed set is empty.
\end{remark}


\begin{definition}\label{def:computability on cantor space}Let $A,B$ be alphabets, we say that a function $f\colon X\subset A^{\NN} \to B^{\NN}$ is \define{computable} if uniformly for every $w\in B^*$ there exists an effectively open set $U_w\subset A^{\NN}$ such that $f^{-1}([w]) = U_w \cap X$.
\end{definition}

Notice that computable functions are always continuous in their domain and that the composition of computable functions is computable. An equivalent and perhaps more intuitive way to think about computable functions on $X\subset A^{\NN}$ is as those for which there is an algorithm that given a prefix of $x \in X$ yields a prefix of $f(x)$, and the sequence of these prefixes monotonously converges to $f(x)$.

\begin{proposition}\label{prop:caracterization_of_computable_OD_function}
    Let $A,B$ be alphabets. A function $f\colon X\subset A^{\NN} \to B^{\NN}$ is computable if and only if there exists a partial computable function $g\colon A^* \to B^*$ such that for any $x \in X$, $g(x|_{\{0,\dots,n\}})$ is defined for every $n$, the cylinders $[g(x|_{\{0,\dots,n\}})]$ form a nested sequence and \[ \{f(x)\}  = \bigcap_{n \in \NN} [g(x|_{\{0,\dots,n\}})]. \]
\end{proposition}

\begin{proof}
    Suppose that $f$ is computable, we shall construct a partial computable function $g\colon A^*\to B^*$ with the required properties. As $f$ is computable, there is an algorithm which can list in some order all $(v,w)\in A^*\times B^*$ such that $[v]\cap X \subset f^{-1}([w])$. Set $g(\epsilon) = \epsilon$ and start the listing algorithm. Iteratively, for every pair $(v,w)$ which is listed, compute the largest prefix $v'\in A^*$ for which $g(v')$ is already defined. If $g(v')$ is strictly a prefix of $w$, set $g(v) = w$ and $g(v'')=g(v')$ for every intermediate word $v''$ which has $v'$ as a prefix and is a prefix of $v$. Otherwise, do nothing. The function $g\colon A^*\to B^*$ resulting for this procedure is a partial computable map which satisfies the required properties.
    
    Conversely, given a partial computable function $g\colon A^* \to B^*$ with the properties of the statement and $w \in B^*$, the collection \[L_w = \{v \in A^* : g(v) \mbox{ is defined and } w \mbox{ is a prefix of } g(v)\},\] is recursively enumerable uniformly on $w$. Therefore the collection $U_w = \bigcup_{v \in L_w}[v]$ is uniformly effectively open and because of the assumptions on $g$ it satisfies $f^{-1}([w]) = X \cap U_w$. 
\end{proof}

We shall now provide a few examples.

\begin{example}\label{ex:shift_map}
    The shift map $\sigma \colon A^{\NN} \to A^{\NN}$ given by $\sigma(x)_n = x_{n+1}$ is computable. Indeed the map $g\colon A^* \to A^*$ which deletes the first symbol, that is $g(x_1\dots x_m) = x_2\dots x_m$ is computable and satisfies the required monotonicity condition.
 \end{example}

 \begin{example}
    Consider the map $f\colon \{\symb{0},\symb{1}\}^{\NN} \to \{\symb{0}, \symb{1},\symb{2}\}^{\NN}$ given by \[ f(x)_0 = \begin{cases}
        \symb{0} & \mbox{ if } x_0 = \symb{0}\\
        \symb{1} & \mbox{ if } x_0x_1 = \symb{10}\\
        \symb{2} & \mbox{ if } x_0x_1 = \symb{11}\\
    \end{cases} \mbox{ and }  f(x)_k = \begin{cases}
        f(\sigma(x))_{k-1} & \mbox{ if } x_0 = \symb{0}\\
        f(\sigma^2(x))_{k-1} & \mbox{ if } x_0 = \symb{1}\\
    \end{cases} \mbox{ for } k\geq 1.\]
    Where $\sigma$ is the shift map from~\Cref{ex:shift_map}. It is clear by its recursive definition that $f$ is a computable map. Moreover, this map provides an example of a computable homeomorphism (computable bijection with computable inverse) between $\{\symb{0},\symb{1}\}^{\NN}$ and $\{\symb{0}, \symb{1},\symb{2}\}^{\NN}$. This construction can be easily generalized to provide a computable homeomorphism between $A^{\NN}$ and $B^{\NN}$ for any finite $A,B$ with at least two elements each, in fact, we shall show that a much more general statement holds in~\Cref{thm:compacto-y-cero-dimensional-tiene-copia-en-el-cantor}.
 \end{example} 


\begin{definition}
    Let $\Gamma$ be a finitely generated group and $X\subset A^{\NN}$. We say that an action $\Gamma \curvearrowright X$ is \define{computable} if for some finite set of generators $S$ and each $s\in S$ we have that the group action map $f_s \colon X \to X$ given by $f_s(x) = sx$ is computable. 
\end{definition}

Notice that as the composition of computable maps is computable, we have that for any $g \in \Gamma$ the map $f_g\colon X \to X$ given by $f_g(x)=gx$ is computable. Therefore, if this condition is satisfied by some generating set, then it is uniformly satisfied by all of them.

Although there are only countably many computable maps, almost every action that one may come up with naturally (without introducing explicitly an uncomputable parameter) is likely going to be computable. Let us provide a few examples of computable actions on Cantor spaces.

\begin{example}
    Let $A = \{\symb{0},\symb{1}\}$ and consider the \define{binary odometer} action $\ZZ \curvearrowright \{\symb{0},\symb{1}\}^{\NN}$ induced by the homeomorphism $T\colon \{\symb{0},\symb{1}\}^{\NN} \to \{\symb{0},\symb{1}\}^{\NN}$ defined by 
    
    \[ T(x)_n = \begin{cases} \symb{0} & \mbox{ if } n < c,\\ \symb{1} & \mbox{ if } n = c,\\
    x_n & \mbox{ if } n > c    
    \end{cases}\qquad \mbox{ where } \quad c = \inf\{k \in \NN: x_k = \symb{0}\}.\]

    Clearly $\{\symb{0},\symb{1}\}^{\NN}$ is effectively closed. We claim that the map $T$ is computable: let $x_0\dots x_{n-1}$ be the first $n$ digits of $x\in A^{\NN}$, we can read the sequence from left to right swapping $\symb{1}$ for $\symb{0}$ until we either read a $\symb{0}$, in which case we swap it for a $\symb{1}$ and leave the rest of the sequence untouched, or we reach the end of the sequence. The resulting sequence corresponds to the first $n$-digits of the image of $T(x)$. A similar argument shows that $T^{-1}$ is computable as well, and as $\{-1,1\}$ generates $\ZZ$, we conclude that the binary odometer action is computable. 
\end{example}

\begin{example}\label{ex:thompson}
Let $\mathcal{U} = \{u_1,\dots,u_n\}$ and $\mathcal{V} = \{v_1,\dots,v_n\}$ be two sets of binary words which induce partitions of $\{\symb{0},\symb{1}\}^{\NN}$, that is, such that $\{\symb{0},\symb{1}\}^{\NN} = \bigcup_{i=1}^{n}[u_i] = \bigcup_{i=1}^{n}[v_i]$ and the unions are disjoint. Consider the homeomorphism $f$ of $\{\symb{0},\symb{1}\}^{\NN}$ such that $f(u_ix) = v_ix$, that is, if an infinite word begins by $u_i$, it replaces that prefix by $v_i$ and leaves the rest of the word unchanged. Notice that for any choice of $\mathcal{U},\mathcal{V}$ the associated homeomorphism is computable.

Three celebrated groups can be defined using these homeomorphisms, namely:
\begin{enumerate}
    \item Thompson's group $F$ of all such homeomorphisms where the elements of both $\mathcal{U}$ and $\mathcal{V}$ are indexed in lexicographical order.
    \item Thompson's group $T$ of all such homeomorphisms where the elements of both $\mathcal{U}$ and $\mathcal{V}$ are indexed in lexicographical order up to a cyclical permutation.
    \item Thompson's group $V$ of all such homeomorphisms.
\end{enumerate}
The three groups above satisfy $F \leqslant T \leqslant V$ and are finitely generated~\cite{CannonFloydParry_thompsongroups}. Their natural actions on $\{\symb{0},\symb{1}\}^{\NN}$ are examples of computable actions on an effectively closed set.
\end{example}

Naturally, for any named set $\mathcal{N}$, we can consider computability of maps from $A^{\mathcal{N}}$ to $B^{\mathcal{N}}$ through the identification of $\mathcal{N}$ with $\NN$. For instance, if we let $\varphi \colon \NN \to \ZZ$ be the bijection $\varphi(n) = (-1)^n \lceil \frac{n}{2}\rceil$, a set $X\subset A^{\ZZ}$ would be effectively closed if and only if $X' = \{x' \in A^{\NN}: x' = x \circ \varphi \mbox{ for some } x \in X\}$ is effectively closed in $A^{\NN}$.

\begin{example}
    Consider the homeomorphisms $t$ and $a$ on $\{\symb{0},\symb{1}\}^{\ZZ}$ given by \[ t(x)_{n} = x_{n+1} \quad \mbox{ and } \quad a(x)_n = \begin{cases}
        \symb{1}-x_n & \mbox{ if } n = 0,\\
        x_n & \mbox{ if } n \neq 0
    \end{cases} \quad \mbox{ for every } x \in \{\symb{0},\symb{1}\}^{\ZZ} \mbox{ and } n \in \ZZ.   \]
    In words, the map $t$ shifts a bi-infinite binary sequence to the left, while the involution $a$ swaps the symbol at the origin. The maps $t,t^{-1}$ and $a$ generate the \define{lamplighter group}. If we identify $\{\symb{0},\symb{1}\}^{\ZZ}$ with $\{\symb{0},\symb{1}\}^{\NN}$ through a computable bijection between $\NN$ and $\ZZ$ we get that the maps corresponding to $t,t^{-1}$ and $a$ are computable. It follows that the induced natural action $(\ZZ \wr \ZZ/2\ZZ)\curvearrowright \{\symb{0},\symb{1}\}^{\ZZ}$ is computable.
\end{example}

\section{Computability and dynamical systems}\label{sec:comp}

In this section we introduce and discuss the main notion of this work: effective dynamical systems. This notion generalizes the idea of a computable action over an effectively closed subset of $A^{\NN}$ to metric spaces with positive topological dimension. Our definition will be expressed in the formalism of computable analysis, to which we will provide a short introduction.

\subsection{Computable analysis}
Computable analysis is about making algorithms able to process infinite objects, such as real numbers or infinite sequences. A convenient unifying way to formalize this idea is to see these objects as points of a metric space with an extra computable structure. Most of the results presented here are well known and can be found in the literature (we refer to ~\cite{Brattka2021} for a modern exposition). We also present some new results related to zero-dimensional spaces. 

\begin{definition}
A \define{computable metric space} is a triple $(X,d,\mathcal{S})$, where $(X,d)$ is a separable metric space
and $\mathcal{S} = \{s_i : i \ge 0\}$ a countable dense subset of $X$, such that there exists an algorithm
which, upon input $(i,j,n) \in \mathbb{N}^{3}$,
outputs  $r\in \mathbb{Q}$ such that
$$|d(s_i,s_j)-r| \le 2^{-n}.$$
We say that the distances~$d(s_i,s_j)$ are \define{uniformly computable} from $i,j$.
\end{definition}

For a rational number $r>0$ and $x$ an element of $X$, we denote by $B(x,r) = \{ z \in X \ : \ d(z,x)<r\}$ the open ball with center $x$ and radius $r$. The balls centered on elements of
$\mathcal{S}$ with rational radii are called \define{basic balls}. A computable enumeration of the basic balls $B_{n}=B(s^{(n)},r^{(n)})$ can be obtained by taking for instance a bi-computable bijection $\varphi\colon\mathbb{N} \rightarrow
\mathbb{N} \times \mathbb{Q}$ and letting $s^{(n)} = s_{\varphi_1 (n)}$ and $r^{(n)} = \varphi_2 (n)$, where $\varphi(n) = (\varphi_1 (n),\varphi_2 (n))$.  We fix such a computable enumeration from now on. For any
subset $I$ (finite or infinite)
of $\mathbb{N}$, we define
\[U_{I} = \bigcup_{n \in I} B_n.\]

An open set $V\subset X$ is \define{effectively open} if $V=U_I$ for some recursively enumerable set $I\subset\mathbb{N}$.  We note that finite intersections, and unions of countably many uniformly effectively open sets, are again effectively open. A set $K\subset X$ is \define{effectively closed} if it is the complement of an effectively open set. 

A point $x\in X$ is \define{computable} if the set~$\{n\in\mathbb{N}:x\in B_n\}$ is recursively enumerable. Note that a point $x\in X$ is computable if and only if one can uniformly compute a sequence $(s_{\varphi(n)})_{n\in\N}$ of elements of $\mathcal{S}$ which satisfy $|s_{\varphi(n)}-x|\leq 2^{-n}$ for every $n \in \NN$. 

\begin{example}
The set $\mathbb{R}$ of real numbers, with the Euclidean metric~$d(x,y)=|x-y|$
and $\mathcal{S}=\mathbb{Q}$ is a computable metric space. The basic balls here are the
open intervals with rational endpoints. In this case the notion of computable point can be split into two: a real $x$ is \define{lower semi-computable} if the set $\{q\in \mathbb{Q}:q<x\}$ is recursively enumerable and \define{upper semi-computable} if $-x$ is lower semi-computable. It is an easy exercise to verify that $x$ is computable if and only if it is both lower and upper semi-computable. An example of a lower semi-computable real number which is not computable is given by \[h=\sum_{e\in \texttt{Halt}}2^{-e}\] where $\texttt{Halt}$ is the halting set (or any other r.e. set that is not decidable). 
\end{example}

\begin{example}\label{ex:closed rational ball}
    For a computable point $x\in X$, a closed ball of the form $\overline{B}(x,r)=\{y\in X : d(x,y)\leq r\}$ is effectively closed if and only if $r$ is upper semi-computable. In particular, the closure of every basic ball is effectively closed. However, if we let $r$ be a lower semi-computable but not computable number, then the open ball $B(x,r)$ is an example of an effectively open set whose closure is not effectively closed.
\end{example}

\begin{example}\label{example:computable_in_cantor_space}
Let $A$ be a countable set with at least two elements. The space $A^{\NN}$ with the metric introduced in~\Cref{subsec:shift_spaces} has a natural computable metric space structure using a standard enumeration of the dense countable set of eventually constant configurations \[\mathcal{S} = \{w\symb{a}^{\infty}:
w \in A^{*}, \symb{a}\in A\}.\] In this case, the basic balls are
the cylinders and thus for the case when $A$ is finite we recover the computable structure introduced in ~\Cref{subsec:shift_spaces}. We will call this the \define{canonical computable structure} of Cantor spaces.  In the case where $A = \NN$ we obtain an analogous canonical computable structure for the \define{Baire Space} $\NN^{\NN}$.
\end{example}

\begin{example}\label{product}
    Let $(X_i,d_i,\mathcal{S}_i)_{i \in \NN}$ be a sequence of uniformly computable metric spaces such that the distances $d_i$ are uniformly bounded. Write $\mathcal{S}_i = \{s_{i,k}\}_{k \in \NN}$. The product space $\left(\prod_{i \in \NN} X_i,d,\mathcal{S}\right)$ is a computable metric space where $d((x_i)_i,(y_i)_i)=\sum_{i \in \NN}2^{-i}d_i(x_i,y_i)$ and \[\mathcal{S}=\{ (x_i)_{i \in \NN}\in  \prod_{i \in \NN} \mathcal{S}_i : \mbox{ there is } m,L \in \NN \mbox{ such that for } \ell \geq L, x_{\ell} = s_{\ell,m} \}.\] 
\end{example}


Let $(X,d)$ and $(X',d')$ be computable metric spaces. Let~$(B'_m)_{m\in\NN}$ denote the canonical enumeration of the basic balls of~$X'$. A function $f\colon X \to X'$ is
\define{computable} if there exists an algorithm
which, given as input some integer $m$,
enumerates a set $I_m\subset \NN$
such that \[f^{-1} (B'_m) = U_{I_m},\]
that is, if the preimages~$f^{-1}(B'_m)$ of basic balls are effectively open sets uniformly in $m$. More generally, for any $Y\subset X$, we say $f\colon Y\to X'$ is \define{computable} if there exists an algorithm
which, given as input some integer $m$,
enumerates a set $I_m$ such that \[f^{-1} (B'_m) = U_{I_m}\cap Y.\]

\begin{remark}
For Cantor spaces $A^{\NN},B^{\NN}$ endowed with the canonical computable structure given in~\Cref{example:computable_in_cantor_space}, the notion of computable function boils down to~\Cref{def:computability on cantor space}.
\end{remark}

It follows immediately from the definition that computable functions are continuous and closed under composition. It is perhaps more intuitively familiar, and in fact equivalent, to think of a computable function as one for which there is an algorithm which, provided with arbitrarily good approximations, outputs arbitrarily good approximations of $f(x)$. The formal statement and the proof is analogous to~\Cref{prop:caracterization_of_computable_OD_function}.

\subsubsection{Computability of compact sets}

\begin{definition}\label{def:recursive compactness}
    A compact set $K\subset X$ is said to be \define{recursively compact}
if the inclusion $$K \subset U_I,$$
where $I$ is some finite subset of $\mathbb{N}$, is semi-decidable. That is, if there is an algorithm which, given $I$ as input, halts if and only if
the inclusion above is verified. 
\end{definition}

We remark that for a recursively compact set $K$, the inclusion $K\subset U_{I}$ can be semi-decided for any effectively open set $U_{I}$, as it suffices to run in parallel the above algorithm on all finite subsets $I' \subset I$ and halt if one of them halts.

\begin{example}
The Cantor space is recursively compact. Indeed, recursive compactness for Cantor spaces boils down to check whether a finite list of cylinders covers the whole space, which is clearly semi-decidable (actually, even decidable, see~\Cref{rem: Cantor is rec.compact}). A similar reasoning applies to the unit interval, which is also recursively compact.  
\end{example}

We will next prove that the product of a collection of uniformly recursively compact metric spaces is again recursively compact. We refer to~\cite{Rett2013} for a treatment of the computability of Tychonoff's theorem in the general case. 

\begin{proposition}\label{thm:Tychonoff_computable}
    Let $(X_i,d_i,\mathcal{S}_i)_{i \in \NN}$ be a uniform sequence of computable metric spaces with uniformly bounded distances. If $K_i \subset X_i$ is uniformly recursively compact, then their product $\prod_{i \in \NN} K_i $ is recursively compact in $\prod_{i \in \NN} X_i$. 
\end{proposition}

\begin{proof} To see that $\prod_{i \in \NN} K_i$ is recursively compact, notice that the projection of a rational ball in the product space to coordinate $i$ contains the whole $X_i$ for all $i>N$ for some sufficiently large $N$ and we can compute $N$ from the radius of the ball. Thus, checking whether a finite list of balls covers $\prod_{i \in \NN} K_i$ boils down to check, for finitely many $i$'s, whether a finite list of balls in $X_i$ covers $K_i$, and this can be done by the uniform recursive compactness of the $K_i$. 
\end{proof}

\begin{remark} Note that the product topology on a uniform sequence of computable metric spaces has the property that the projection onto a uniform (in particular a finite) collection of coordinates is a computable function.
\end{remark}

We will make use of the following result, which is a computable version of the fact that in a Hausdorff space, a subset of a compact set is compact if and only if it is closed. 

\begin{proposition}
\label{prop.equivalence.comp.closed.subset}
Let $(X,d,\mathcal{S})$ be a recursively compact computable metric space. A subset $K \subset X$ is recursively compact if and only if it is effectively closed.
\end{proposition}

\begin{proof}
Assume $K$ is effectively closed. Then $X\setminus K$ is effectively open. Let $U$ be a finite union of open basic balls. To see that we can semi-decide whether $K\subset U$, simply note that $U$ covers $K$ exactly when $U \cup (X\setminus K)$ covers $X$, which can be semi-decided since $X$ is recursively compact. Conversely, if $K$ is recursively compact, then we can enumerate its complement by enumerating the complement of the closure of all the finite unions of basic balls that cover $K$. 
\end{proof}

\begin{remark}\label{rem_emptiness}
    An important consequence of the previous proposition is that in any recursively compact space $X$ there is an algorithm which, given as input a description of an effectively closed set $K$ and an effectively open set $U_I$, halts if and only if $K\subset U_I$. In particular, we can uniformly semi-decide if $K = \varnothing$. 
\end{remark}

We also note that if $f\colon X\to X'$ is a computable function between computable metric spaces, then the image of a recursively compact set is recursively compact uniformly in the description of the set. We end this section with the following two important propositions. 

\begin{proposition}
\label{prop.inverse.of-computable-function-on-compatc-space}
    Let $f\colon X \to Y$ be a function between computable metric spaces that is computable on a recursively compact set $K\subset X$. Then uniformly for any basic ball $B\subset X$ there exists an effectively open set $U_{B} \subset Y$ such that $f(B\cap K) = f(K) \cap U_B$. In particular, if $f$ is injective it has a computable inverse $f^{-1}\colon f(K)\to K$.  
\end{proposition}
\begin{proof}
 Let $B$ be a basic ball. We show how to uniformly compute an effectively open $U_B\subset Y$ such that $f(B\cap K) = U_B \cap f(K)$.  As $K$ is recursively compact, we have that $K\setminus B$ is recursively compact, and thus $f(K\setminus B)$ is recursively compact in $Y$ uniformly on $B$. The set $U_B=Y\setminus f(K\setminus B)$ is effectively open and satisfies the required property. 
\end{proof}



\begin{proposition}\label{prop:computability_of_fixed_points}
    Let $f\colon K\to K$ be a computable function on an effectively closed subset of a computable metric space $X$. Then the set
    \[\{x\in K : f(x)=x\}\]
    is effectively closed.
\end{proposition}
\begin{proof} Simply note that the function $F(x) = d(x,f(x))$ is computable on $K$, and therefore 
    \[
    X\setminus \{x\in K : f(x)=x\} = (X\setminus K) \cup F^{-1}(\{r\in \RR: r>0\}) 
    \]
    is an effectively open set. 
\end{proof}

\subsection{Computability for zero-dimensional spaces}

Another interesting application of \Cref{prop.equivalence.comp.closed.subset} is the construction of a basis of clopen sets for any given recursively compact zero-dimensional subset of a computable metric space. 

\begin{proposition}\label{prop.clopens}
    Let $K$ be a zero-dimensional recursively compact subset of a computable metric space $X$. Then one can uniformly compute a collection $(C_n)_{n \in \NN}$ of finite unions of basic balls in $X$ such that $(C_n\cap K)_{n \in \NN}$ forms a basis of clopen sets for $K$. 
\end{proposition}
\begin{proof}
    Since $K$ is zero-dimensional, it has a basis of clopen sets. Let $C$ be clopen in $K$. Since $C$ is open in $K$, there exists a set $U$ open in $X$ such that $C = U\cap K$. Since $C$ is also closed and $K$ is compact, $C$ is compact as well, and thus the set $U$ can be taken to be a finite union of basic balls. Now consider the collection $V_1, V_2, \dots$ of all finite unions of basic balls. By what we have just shown we know that $(V_i\cap K)_{i \in \NN}$ contains a basis of clopen sets for $K$. We claim that we can uniformly semi-decide, given a finite union $V$, whether $C = V \cap K$ is clopen in $K$. Indeed, note that if $C$ is a clopen set in $K$, then one has  
    \[ C = \overline{C} = \overline{V}\cap K\]
where $\overline{V}$ is the union of the closures of the finitely many balls defining $V$. This is equivalent to have 
\[\left(K\setminus C\right)  \subset X\setminus \overline{V}.\] 

But $K\setminus C = K\setminus V$ is effectively closed, and therefore recursively compact by \Cref{prop.equivalence.comp.closed.subset}. Since $X\setminus\overline{V}$ is an effectively open set, it follows that this last relation is semi-decidable.  
\end{proof}

It is a well known result of Brouwer that every zero-dimensional compact space without isolated points is homeomorphic to the Cantor space $\{\symb{0},\symb{1}\}^\N$. Here we will prove an effective version of this result. To state it we need the following notion of computability for closed sets. 

\begin{definition}
    A closed set $X$ is \define{recursively enumerable (r.e.)} if one can uniformly compute a sequence $(x_i)_{i \in \NN} \subset X$ of points that is dense in $X$. If $X$ is both effectively closed and r.e. then we say it is \define{computably closed}. 
\end{definition}

\begin{remark}\label{rem:comp-funct-send-re-to-re}
    We note that since computable functions send uniformly computable sequences of points to uniform computable sequences of points, it follows that the image of a closed r.e. set by a computable function is also closed r.e. 
\end{remark}

 \begin{theorem}\label{thm:compacto-y-cero-dimensional-tiene-copia-en-el-cantor} 
        Let $X$ be a nonempty recursively compact zero-dimensional subset of a computable metric space. Then $X$ is computably homeomorphic to an effectively closed subset $E$ of $\{\symb{0},\symb{1}\}^\N$. Moreover, 
        \begin{itemize}
            \item if $X$ is computably closed, then $E$ can be taken to be computably closed,
            \item if $X$ is computably closed and has no isolated points, then $E$ can be taken to be $\{\symb{0},\symb{1}\}^\N$.
        \end{itemize} 
    \end{theorem}

    \begin{proof}
      If $X$ is a singleton, then there is nothing to prove. We assume therefore that $X$ contains at least two points. We claim that one can uniformly compute a sequence $(\mathcal P^n)_{n\in\N}$ of partitions of $X$ made by effectively closed clopen sets, where each $\mathcal P^n$ is indexed as $\{P^n_0,\dots,P^n_{k_n}\}$, such that each set in $\mathcal P^n$ has diameter at most $2^{-n}$, and such that the elements of $\mathcal P^n$ are unions of elements in $\mathcal P^{n+1}$. Indeed, it suffices to observe that, given any $n\in\NN$, by using \Cref{prop.clopens}, one can uniformly compute a basis of clopen sets whose diameter is less than $2^{-n}$ (recall that from the index of a basic ball one can compute its radius). Now, given any clopen and effectively closed set, call it $C$, we can use recursive compactness of $C$ (see \Cref{prop.equivalence.comp.closed.subset}) to find a finite collection of basic clopen sets of small diameter that covers $C$. By iteratively applying this observation, we can compute the sequence of partitions  $(\mathcal P^n)_{n\in\N}$ as needed. Moreover, by our assumption on $X$ we can assume that $\mathcal{P}^0$ has more than one element. Now, for each $n\in\N$ we let $B_n=\{0,\dots,k_n\}$. It follows that $B_n$ has at least two elements. As $(k_n)_{n\in\N}$ is a computable sequence of natural numbers, it follows that $\prod_{n \in \NN} B_n$ is an effectively closed and recursively compact subset of $\N^\N$ by~\Cref{thm:Tychonoff_computable}. We now consider the following subset of $\prod_{n \in \NN} B_n$:
        \[Y=\{y\in\prod_{n \in \NN} B_n : \text{ for all } n\in\N, P^n_{y(n)}\cap X\neq\varnothing \text{ and } P^n_{y(n)}\supset P^{n+1}_{y(n+1)}\}. \] 
        
        Observe that $Y$ is effectively closed. Indeed, we can semi-decide whether $P^n_{y(n)}\cap X=\varnothing$ by the recursive compactness of $X$. The condition $P^k_{y(k)}\supset P^{k+1}_{y(k+1)}$ is decidable by construction. As $\prod_{n \in \NN} B_n$ is a recursively compact subset of $\N^\N$, it follows that $Y$ is recursively compact as well by \Cref{prop.equivalence.comp.closed.subset}. 

        We now observe that $Y$ and $X$ are computably homeomorphic. We define the function  $\psi\colon Y\to X$ by the following equality:
        \[\{\psi (y)\}=\bigcap_{n\in\N} P^n_{y(n)}.\]
        This function is well defined, as the intersection of a nested sequence of compact sets with diameters tending to zero must be a singleton. It is clear from the expression above that $\psi$ is continuous, computable and bijective. It follows from \Cref{prop.inverse.of-computable-function-on-compatc-space} that the inverse of $\psi$ is computable, and thus it is a computable homeomorphism.

        We shall now prove that $\prod_{n \in \NN} B_n$ is recursively homeomorphic to $\{\symb{0},\symb{1}\}^\N$. For this we take $T\subset\N^*$ as the set of words $\{\varepsilon\} \cup \bigcup_{n\in\N}B_0\times\dots \times B_n$. Then $T$ is an infinite tree (a set of words closed by prefix) with the property that each word can be extended by at least two different words on the right. Furthermore, $T$ is a decidable subset of $\N^*$ and $\prod_{n \in \NN} B_n$ can be identified with the set of infinite rooted paths in this tree. 

        
        We define a function $h\colon T\to \{\symb{0},\symb{1}\}^\ast$ in the following recursive manner. First, $h(\epsilon)=\epsilon$. Now let us assume that $h(w)$ has been defined for some $w\in T$. Then we compute the set of words $\{w_0,\dots,w_m\}$ which have $w$ as prefix, and whose length is equal to the length of $w$ plus one, labeled in lexicographical order. As each word in $T$ can be extended in two ways on the right, this set is nonempty and $m \geq 1$. Let
        \[v_0=\symb{0},v_1=\symb{10},v_2=\symb{110},\dots,v_{m-1} = \symb{1}^{m-1}\symb{0}, v_m=\symb{1}^{m},\]
        and define $h(w_i)=h(w)v_i$, for $0\leq i\leq m$. These conditions define $h(w)$ for every $w\in T$. From our construction it is clear that $h$ is a computable function, it is monotone for the prefix order, and the length of $h(w)$ tends to infinity with the length of $w$. Thus $h$ induces a function
        \[H\colon \prod_{n \in \NN} B_n\to \{\symb{0},\symb{1}\}^\N\]
        Given by $H(y)(n)=h(y_0\dots y_k)(n)$, the $n$-th element in the word $h(y_0\dots y_k)$, for some $k$ big enough and computable from $n$. A routine verification shows that $H$ is a computable homeomorphism between $\prod_{n \in \NN} B_n$ and $\{\symb{0},\symb{1}\}^\N$.

        The fact that $X$ is computably homeomorphic to an effectively closed subset $E$ of $\{\symb{0},\symb{1}\}^\N$ follows by composing the computable homeomorphism between $X$ and $Y\subset \prod_{n \in \NN} B_n$, and the computable homeomorphism between $\prod_{n \in \NN} B_n$ and  $\{\symb{0},\symb{1}\}^\N$. Computable homeomorphisms preserve the property of being effectively closed, so $E$ will be effectively closed when $X$ is effectively closed. The same holds for the property of being r.e. closed by~\Cref{rem:comp-funct-send-re-to-re}. 

        We now prove that if $X$ is computably closed, nonempty, and has no isolated points, then it is computably homeomorphic to $\{\symb{0},\symb{1}\}^\N$. For this purpose we make some modifications to the proof given above. We must add some conditions to the sequence $(\mathcal P^n)_{n\in\N}$. The first extra condition is that for every $n\in\N$, every element in $\mathcal P^n$ intersects $X$. As $X$ is computably closed, we can decide which elements from $\mathcal P^n$ have empty intersection with $X$, and discard those elements. The second extra condition is that for every $n\in\N$ every element in $\mathcal P^n$ contains at least two elements in $\mathcal{P}^{n+1}$. This can be obtained by computably taking a sub-sequence of the original sequence of partitions. The existence being assured by the fact that $X$ has no isolated points.

        After these considerations, the set $Y$ is defined in the same manner. Observe that being a computable homeomorphic image of a computably closed set $X$, $Y$ is computably closed. Now let $T$ be the set of words (including the empty word) $w\in \N^\ast$ such that the cylinder $[w]\subset\N^\N$ has nonempty intersection with $Y$. As $Y$ is computably closed, is follows that $T$ is a computable tree. Observe that from our choice of $\mathcal{P}^n$ every element in the tree has at least two children. Then we just repeat the construction of $h$ given above on the tree $T$. The function $H$ constructed before induces now a computable homeomorphism between $Y$ and $\{\symb{0},\symb{1}\}^\N$.\end{proof}

\subsection{Effective subshifts}\label{computability on shift spaces}
Let us recall that a subshift $X\subset A^\Gamma$ is a closed subset which is invariant by the shift action. A subshift $X\subset A^{\ZZ}$ is said to be effective when it is an effectively closed subset of $A^\Z$, that is, when the set of all words not appearing in the subshift is recursively enumerable. This notion has been extensively used in the literature  \cite{Cenzer2008,Hochman2009b,AubrunSablik2010,DurandRomashchenkoShen2010}. Our goal in this section is to generalize this notion from $\Z$ to an arbitrary finitely generated group $\Gamma$. When $\Gamma$ has decidable word problem, there is a canonical way to endow the full $\Gamma$-shift $A^\Gamma$ with a computable metric structure such that one can speak about effectively closed sets (and therefore effective subshifts). However, as we will see, for more general groups the task is less straightforward.


\subsubsection{Computable structure on $A^\Gamma$ for a group with decidable word problem}

Let $\Gamma$ be a finitely generated group with decidable word problem. Then there is a bijection (see~\Cref{rem:decidablewp_isnamed}) $\nu\colon\N\to \Gamma$ for which the pullback of the group operation to $\N^2\to \N$ becomes a computable function. We now introduce a computable metric space structure on $A^\Gamma$. For this purpose we consider the function  
\[\psi\colon A^\N\to A^\Gamma\]
\[(x_i)_{i\in\N}\mapsto (x_{\nu^{-1}(g)})_{g\in\Gamma}\]
which is in fact an homeomorphism, and let $\mathcal S=\{\psi(w\symb{a}^\infty) : w\in A^\ast, \symb{a}\in A \}$ to be the collection of basic points in $A^\Gamma$. The distance $d$ on $A^\Gamma$ obtained by declaring $\psi$ to be an isometry generates the prodiscrete topology. In this way, $\mathcal{S}$ becomes a dense subset of $A^\Gamma$ on which the distance $d$ is clearly computable. 

\begin{remark}
We note that this computable metric structure for $A^\Gamma$ makes the shift action $\Gamma\curvearrowright A^\Gamma$ computable. In order to see this, let $g\in \Gamma$, and let $f_g\colon A^\Gamma\to A^\Gamma$ be defined by $f_g(x)=gx$. To show that $f_g$ is computable, it is enough to note that $\psi^{-1} f_g\circ \psi\colon A^\N \to A^\N$ is computable. Indeed, this function can be written as $(x_i)_{i\in\N}\to (x_{h(i)})_{i\in\N}$, where $h$ is a computable bijection of $\N$ which comes from the computability of the group operation. 
\end{remark}

\begin{remark}
The computable metric space structure produced in this way for $A^\Gamma$ is inherent to $\Gamma$ and does not depend on the choice of $\nu$. Indeed, let  $\nu'$ be another bijection  $\N\to \Gamma$ as before, and let us denote by $(A^\Gamma,d,\mathcal S)$, $(A^\Gamma,d',\mathcal S')$ the computable metric space structures associated to $\nu$ and $\nu'$. Then $\operatorname{id}\colon (A^\Gamma,d,\mathcal S)\to (A^\Gamma,d',\mathcal S')$ is a computable homeomorphism. This can be proved by noting that  $\nu ^{-1} \circ \nu' $ is a computable bijection of $\N$. For details, the reader is referred to \cite[Section 2.4 and Section 5]{Carrasco_nicanor_subgroup_membership_2023}).
\end{remark}

Once $A^\Gamma$ has been endowed with its canonical computable structure, one can naturally extend the definition of effective subshift for $A^\Z$ to make sense in $A^\Gamma$ by simply requiring the subshift to be effectively closed. Note that, by definition, a set $X\subset A^\Gamma$ is effectively closed if and only if the set of patterns $p$  such that $[p]\cap X=\varnothing$ is recursively enumerable. 

We now present a construction that allows us to have a useful notion of effective subshift when $\Gamma$ is an arbitrary finitely generated group.


\subsubsection{Computable structure on $A^\Gamma$, for arbitrary finitely generated groups}

We shall prove later that for some groups, the space $A^\Gamma$ can \emph{not} be endowed with a computable metric space structure for which the action by translations $\Gamma\curvearrowright A^\Gamma$ is computable. However, the set $A^\Gamma$ can always be identified with a closed subset of $A^{F}$ for a suitable free group $F$. The idea is to use the fact that finitely generated free groups have decidable word problem, which as we have just seen allows us to endow $A^F$ with a canonical computable structure. Now to the details. 

Let $\Gamma$ be a finitely generated group, which is not assumed to have decidable word problem. Let $S$ be a finite set of generators of the group $\Gamma$. Consider the free group $F(S)$ generated by $S$, and consider the canonical epimorphism $\phi\colon F(S)\to \Gamma$, where every reduced word on $S$ is mapped to its corresponding element in $\Gamma$. Let $\widehat{\phi} \colon A^\Gamma \to A^{F(S)}$ be the map given by $\widehat{\phi}(x)(w) = x(\phi(w))$ for every $w \in F(S)$. 

Every subshift $X\subset A^{\Gamma}$ induces a \define{pullback subshift} $\widehat{X}\subset A^{F(S)}$ given by
\[\widehat{X} = \{ \widehat{\phi}(x) \in A^{F(S)} : x \in A^{\Gamma}\}.\]


\begin{definition}\label{def:effective_subshift_through_free_group}
    We say that a subshift $X\subset A^{\Gamma}$ is \define{effective} if for some finite set $S$ of generators of $\Gamma$ the pullback subshift $\widehat{X}\subset A^{F(S)}$ is an effectively closed subset of $A^{F(S)}$.
\end{definition}

It can be seen that this definition does not depend on the set of generators $S$. Indeed, let $S'$ be another set of generators and denote by $\widehat{X}'$ its pullback in $A^{F(S')}$. For each $s \in S$, let $\psi(s)$ be a word in $S'$ such that $s$ is equal to $\psi(s)$ in $\Gamma$. Then $\psi$ extends to a computable injective homomorphism from $F(S)$ to $F(S')$. It follows that the map $\Psi\colon \widehat{X}\to \widehat{X}'$ given by $\Psi(x)(w) = x(\psi(w))$ is a computable homeomorphism. In particular, $\widehat{X}'$ is also effectively closed. A similar argument shows that for a group with decidable word problem, $\widehat{X}\subset A^{F(S)}$ is effectively closed exactly when  $X\subset A^\Gamma$ is effectively closed, as $\widehat{A^{\Gamma}}$ is then effectively closed in $A^{F(S)}$.

\begin{remark}
In the literature~\cite{ABS2017} there is a different computability notion for shift spaces (\Cref{def:ECP}). While both notions coincide for recursively presented groups (\Cref{cor:ECP_plus_RP_implies_EDS}), it turns out that for general finitely generated groups the two definitions are not equivalent. In~\Cref{sec:computability-on-shift-spaces} we shall provide a characterization of this other notion and clarify the relation with ours. 
\end{remark}

\begin{example}\label{example:fullshift_RP_is_effective}
    If $\Gamma$ is recursively presented, then the full shift $A^{\Gamma}$ is effective. Indeed, as the word problem of $\Gamma$ is recursively enumerable, the set of patterns of the form $p:\{1_{F(S)},w\}\to A $ satisfying $\underline{w}=1_{\Gamma}$ and $p(1_{F(S)})\ne p(w)$ is recursively enumerable. The union of cylinders associated to these patterns in $A^{F(S)}$ is equal to the complement of $\widehat{A^\Gamma}$.
    \end{example}

\begin{remark}\label{rem.effective.subshift.is.EDS}
    Observe that the map $\widehat{\phi}$ is a homeomorphism between $X$ and $\widehat X$. Moreover, we can define an action $\Gamma\curvearrowright \widehat {X}$ by setting $gx=wx$, where $w\in F(S)$ is any element satisfying $\phi(w)=g$. In this manner the actions $\Gamma\curvearrowright \widehat{X}$ and $\Gamma\curvearrowright X$  are topologically conjugate. In particular, we have that if $X$ is an effective subshift, then $\Gamma\curvearrowright X$ is topologically conjugate to a computable action $\Gamma \curvearrowright \widehat{X}$ where $\widehat{X}$ is an recursively compact subset of a computable metric space. This is the property that will allow us to generalize the notion of effective subshift to general dynamical systems. 
\end{remark}

\subsection{Effective dynamical systems}\label{sec.eds}

Let $X$ be a recursively compact subset of a computable metric space. Similarly to the zero-dimensional case, we say that $\Gamma \curvearrowright X$ is a \define{computable action} if for some finite set of generators $S$, we have that for each $s\in S$ the group action map $f_s \colon X \to X$ given by $f_s(x) = sx$ is a computable function. Notice that in this case, $f_g\colon X \to X$ is in fact uniformly computable for all $g\in\Gamma.$ In particular, if an action is computable, it will satisfy the definition with respect to every finite set of generators.

We are interested in the behavior of actions as topological dynamical systems, and thus will consider them up to topological conjugacy. Our focus will be on those conjugacy classes that contain some computable representative.   

\begin{definition} 
    Let $X$ be a compact metrizable space and $\Gamma$ a finitely generated group. We say that an action $\Gamma \curvearrowright X$ is an \define{effective dynamical system (EDS)} if it is topologically conjugate to a computable action $\Gamma \curvearrowright \widehat{X}$ for some recursively compact subset $\widehat{X}$ of a computable metric space.
    \end{definition}

We insist in that our definition of effective dynamical system \emph{does not} require the space $X$ to have any computable structure or the topological conjugacy to be computable. Given an EDS $\Gamma \curvearrowright X$, we will refer to a topologically conjugate instance of a computable action $\Gamma \curvearrowright \widehat{X}$ over a recursively compact subset $\widehat{X}$ of a computable metric space as a \define{computable representative}.


\begin{example}\label{example:effective_subshift_is_EDS}
      An effective subshift $X\subset A^{\Gamma}$ is an EDS. Indeed, if $A^{\Gamma}$ is an effective subshift, then $\Gamma\curvearrowright X$ is topologically conjugate to the action $\Gamma\curvearrowright \widehat{X}$, where $\widehat{X}$ is the subshift from \Cref{def:effective_subshift_through_free_group} (see also \Cref{rem.effective.subshift.is.EDS}). We will see later that, conversely, the only subshifts which are EDS are the effective ones, thus justifying the name (\Cref{prop:EDS_and_effective_subshift_are_equivalent}). We remark that the above holds for any finitely generated group $\Gamma$, even if it is not recursively presented.
\end{example}


The following example shows that the conjugacy class of an effective dynamical system can contain uncountably many systems. In particular, it contains representatives that are not computably homeomorphic.

\begin{example}
Fix some computable angle $\alpha \in [0,2\pi]$ and consider $C_r=\{x\in \RR^2: \|x\|=r\}$ to be the circle of radius $r>0$. Note that  the space $C_r$ is a recursively compact subset of $\RR^2$ only for computable $r$. Despite this, the conjugacy class of the action $\ZZ \curvearrowright C_r$ induced by the rotation by $\alpha$ has a computable representative, namely  $\ZZ \curvearrowright C_1$, and is therefore an EDS regardless of the value of $r$. 
\end{example}

\begin{example}\label{example:torus_matrix_multiplication}
The action $\operatorname{GL}_n(\ZZ) \curvearrowright \RR^n/\ZZ^n$ of the general linear group on the $n$-dimensional torus by left matrix multiplication is an EDS.
\end{example}

\begin{example}
    Recall the Thompson's groups $F$ and $T$ from~\Cref{ex:thompson}. These groups are more often defined in the literature as the space of piecewise linear homeomorphisms of $[0,1]$ (in the case of $T$, of $\RR/\ZZ$) that preserve orientation and whose non-differentiable points are dyadic rationals and whose slopes are all powers of 2. These maps are clearly computable and thus it follows that the natural actions $F\curvearrowright [0,1]$ and $T \curvearrowright \RR/\ZZ$ are EDS.
\end{example}

    Next we are going to show that the class of EDS is closed under computable topological factor maps. We stress the fact that in this result we need both actions to be represented as computable actions in computable metric spaces and the factor to be computable as well. In~\Cref{sec:factors} we will later show that without strong assumptions that enforce computability, the class of EDS is in general not closed under topological factor maps.

    \begin{proposition}\label{prop:computable_factor_EDS_is_EDS}
        Let $\Gamma$ be a finitely generated group and $X$, $Y$ be computable metric spaces. Let $X'\subset X$ be a recursively compact set and $\Gamma\curvearrowright X'$ be a computable action. Let $\Gamma\curvearrowright Y' \subset Y$ be a topological factor of $\Gamma\curvearrowright X'$ by a computable function $f\colon X'\to Y'$. Then $Y'$ is a recursively compact set and $\Gamma\curvearrowright Y'$ is a computable group action.    
    \end{proposition}

    \begin{proof}
         As $X'$ is recursively compact and $f$ is computable, it follows that $Y'$ is recursively compact. Let $(B^X_i)_{i \in \NN}$ and $(B^Y_i)_{i \in \NN}$ be recursive enumerations of the basic balls in $X$ and $Y$ respectively. Let $s \in \Gamma$ and let $i \in \NN$. As $f$ is computable it follows that, uniformly in $i$, there is a recursively enumerable set $I_i \subset \NN$ such that $f^{-1}(B^Y_i) = X' \cap \bigcup_{j \in I_i} B^{X}_j$. As the action is computable, uniformly for each $j \in \NN$ there is a recursively enumerable set $I'_j\subset \NN$ such that \[s^{-1}f^{-1}(B^Y_i) = X' \cap \bigcup_{j \in I_i}\bigcup_{k \in I'_j} B_k^X.\]
        Finally, as $f$ is computable and $X'$ recursively compact, it follows by~\Cref{prop.inverse.of-computable-function-on-compatc-space} that uniformly for $k \in \NN$ there is a recursively enumerable set $I''_k\subset \NN$ such that \[ s^{-1}(B_i^Y\cap Y') = f(s^{-1}f^{-1}(B_i^Y)) = \bigcup_{j \in I_i}\bigcup_{k \in I'_j}\bigcup_{m \in I''_k} B_m^Y.  \]
        Where the first equality holds because $f$ is $\Gamma$-equivariant. It follows the map $y \mapsto sy$ is computable and as $s$ is arbitrary, we obtain that $\Gamma \curvearrowright Y'$ is computable.
    \end{proof}


\section{EDS as factors of computable actions on 0-dimensional effectively closed sets} \label{sec:resultado}
In this section we prove~\Cref{thm:zero_dim_effective_extension}. Let us briefly recall the standard procedure to build a zero-dimensional topological extension of a dynamical system $\Gamma\curvearrowright X$.
\begin{enumerate}
    \item First we extract a sequence $(\mathcal P_n)_{n\in\N}$ of open covers of $X$ with diameters tending to zero, and associate to each one of them a subshift $Y_n$. 
    \item Define the sequence of subshifts $(Z_n)_{n\in\N}$ where $Z_n = \prod_{k \leq n} Y_k$ with the coordinate-wise shift, and consider the projection factor maps $\pi_{n+1}\colon Z_{n+1}\to Z_n$.
    \item Define the inverse limit $Z= \varprojlim Z_n$, this is a topological extension of $\Gamma\curvearrowright X$.
\end{enumerate}
In our proof we effectivize this construction. For this purpose, we first need to prove effective versions of some of the steps. This includes infinite products, inverse limits, and subshifts associated to covers.


\subsection{Effective versions of some classical constructions.}
 
 In this section we prove effective versions of a few classical constructions commonly used in dynamical systems, and which may therefore be of independent interest.


\subsubsection{Products of computable actions and inverse limit constructions}

Now we show that the operations of countable product and inverse limits are computable under mild computability assumptions. For the next two results, $(Y_n)_{n \in \NN}$ will be a sequence of uniformly computable metric spaces as in~\Cref{product} and we will consider $\prod_{n\in\N} Y_n$ with its product computable metric space structure. 
    
    \begin{proposition}\label{prop:countable-product-of-EDS}
        Let $X_n \subset Y_n$ be a sequence of subsets and $\Gamma\curvearrowright X_n$ be a uniform sequence of computable actions of the finitely generated group $\Gamma$. Then the component-wise action
        \[ \Gamma\curvearrowright \prod_{n\in\N} X_n  \subset \prod_{n\in\N} Y_n \]
        is computable. 
    \end{proposition}
    \begin{proof}
        Let us fix $g\in\Gamma$. Let $n\in\N$ and let $U_i$ be an effectively open subset of $X_i$ for $i\in\{0,\dots,n\}$. We verify that the preimage by $g$ of  \[U=U_0\times\dots \times U_n\times Y_{n+1}\times Y_{n+2}\times\dots\] is effectively open in $\prod_{n\in\N}X_n$ and that this process is uniform in $n$ and the $U_i$. Indeed, \[g^{-1}(U) = g^{-1}(U_0)\times\dots \times g^{-1}(U_n)\times X_{n+1}\times X_{n+2}\times\dots\] This set is effectively open in $\prod_{n\in\N}X_n$, and it can be uniformly computed from $n \in \NN$ and a description of $U_0,\dots,U_n$ because each of the $g^{-1}(U_i)$ is effectively open in $X_i$, and can be uniformly computed from a description of $U_i\subset Y_i$.
    \end{proof}

    \begin{proposition}\label{prop:inverse-limit-of-EDS}
        Let $\Gamma\curvearrowright X_n$ be as in the previous statement, where each $X_n$ is now assumed to be an effectively closed subset of $Y_n$. Let $(\pi_n)_{n\geq 1}$ be a sequence of uniformly computable functions, where each $\pi_{n+1}\colon X_{n+1}\to X_n$ is a topological factor map from $\Gamma\curvearrowright X_{n+1}$ to  $\Gamma\curvearrowright X_{n}$. Then the inverse limit
        \[\varprojlim X_n=\{(x_n)\in\prod_{n \in \NN} X_n\mid\pi_{n+1}(x_{n+1})=x_n, n\geq 1\}\]    
        is an effectively closed subset of $\prod_{n \in \NN} Y_n$. 
    \end{proposition}
    \begin{proof}
        Let $f\colon \prod_{n \in \NN} X_n\to \prod_{n \in \NN} X_n$ be the function defined by $(x_n)_{n\in\N}\mapsto (\pi_{n+1}(x_{n+1}))_{n\in\N}$. We claim that $f$ is computable.  We verify that the preimage by $f$ of a set of the form  
        \[U=U_0\times\dots \times U_n\times X_{n+1}\times X_{n+2}\times\dots\]
        is effectively open in $\prod_{n \in \NN} Y_n$, uniformly on $n$ and the $U_i$. Indeed, the preimage $f^{-1}$ can be written as 
        \[ f^{-1}(U)=\pi_1^{-1}(U_0)\times\dots \times \pi_{n+1}^{-1}(U_n)\times X_{n+1}\times X_{n+2}\times\dots\]
        This set is effectively open uniformly on the $U_i$ because  $(\pi_n)_{n\geq 1}$ is a sequence of uniformly computable functions. On the other hand, it is clear that $\varprojlim X_n$ equals the set of fixed points of $f$ in $\prod_{n \in \NN} X_n$. Then it follows from \Cref{prop:computability_of_fixed_points} that  $\varprojlim X_n$ is effectively closed. 
    \end{proof}


	\subsubsection{Effective covers and partitions, and their associated subshifts}\label{subsec:subshift-covers}

    In this subsection we fix a recursively compact set $X$ which lies in a computable metric space $\mathcal{X}$. We start by reviewing some standard terminology regarding covers and the construction of a subshift from a cover. A \textbf{cover} $\mathcal P$ of $X$ is a finite collection of sets whose union equals $X$. A cover $\mathcal P$ is said to be open (resp.~closed) if it consists on sets which are open in $X$ (resp.~closed in $X$). We say that the cover $\mathcal P'$ \textbf{refines} $\mathcal P$ if every element in $\mathcal P'$ is contained in some element of $\mathcal P$. The join of two covers $\mathcal P \lor \mathcal P'$ is the cover $\{P\cap P'\mid P\in \mathcal P, P'\in\mathcal P'\}$. Given a group action $\Gamma\curvearrowright X$ and a finite subset $F\subset \Gamma$, we write $\bigvee _{g\in F} g^{-1}\mathcal P$ for the join of all $g^{-1}\mathcal{P}$ for $g\in F$. The \textbf{diameter} of the cover $\mathcal P$ is the maximum of the diameters of its elements. A cover $\mathcal P$ of $X$  is said to be \textbf{generating} for the group action $\Gamma\curvearrowright X$ if for each $\varepsilon >0$ there is a finite subset $F\subset \Gamma$ such that $\bigvee _{g\in F} g^{-1}\mathcal P$ has diameter at most $\varepsilon$. 
	
    Given a cover of $X$ labeled as $\mathcal{P}=\{P_0,\dots,P_n\}$, we define the \textbf{subshift cover} $Y(\Gamma\curvearrowright X, \mathcal{P})$ by \[ Y(\Gamma\curvearrowright X, \mathcal{P}) = \{ y \in \{0,\dots,n\}^{\Gamma}  : \mbox{ there exists } x \in X \mbox{ such that for every } g \in \Gamma, g^{-1}x \in \overline{P_{y(g)}}\}.   \] 
    
	The idea behind this definition is that every configuration in the subshift labels the orbit of some element $x\in X$ under the action $\Gamma\curvearrowright X$ by indicating the elements of $\mathcal{P}$ it hits. A compactness argument shows that a configuration $y$ lies in $Y(\Gamma \curvearrowright X, \mathcal P)$ if and only if for every finite $F\subset \Gamma$ and pattern $p\colon F\to \{0,\dots,n\}$ which occurs in $y$, the set
	  \[ D(p):=\bigcap_{g\in F} g^{-1}(\overline{P_{p(g)}}). \]
	is nonempty. In particular, $Y(\Gamma\curvearrowright X, \mathcal{P})$ is indeed a subshift. Now we prove that this well known construction is computable when the cover is made by effectively closed sets.
     \begin{definition}
        A cover $\mathcal P$ of $X$ is called \textbf{effective} if its elements are effectively closed sets.
    \end{definition}
	\begin{proposition}\label{prop:effective-covers-give-effective-subshifts}
        Let $\Gamma \curvearrowright X$ be a computable action of a finitely generated group, $S$ be a finite generating set of $\Gamma$, and let $F(S)\curvearrowright X$ be the induced action. Then the associated subshift $Y(F(S) \curvearrowright X, \mathcal{P})$ is effective uniformly for all effective covers $\mathcal{P}$.
	\end{proposition}
    \begin{proof} It suffices to show that we can semi-decide, given an effective cover $\mathcal{P}$ and a pattern $p \colon W \to F(S)$ with $W\subset F(S)$ finite, whether $D(p)=\varnothing$. Observe that $D(p)$ is an effectively closed subset of $X$, uniformly in $p$ and $\mathcal{P}$. Indeed, each $P_i$  is an effectively closed subset of $X$, the preimage of an effectively closed set by a computable function is effectively closed, and the finite intersection of effectively closed sets is effectively closed. Thus, we can uniformly semi-decide whether ${D(p)}$ is empty using the recursive compactness of $X$ (see \Cref{def:recursive compactness} and \Cref{rem_emptiness}).
    \end{proof}    
     
    \begin{remark}
    In what follows we focus on the construction of effective covers. These will be obtained intersecting $X$ with unions of closures of basic balls in $\mathcal{X}$. Note that $X$ is not assumed to be a computable metric space, and this is the reason why we appeal to the space $\mathcal X$.   
    \end{remark}
    \begin{proposition}\label{prop:Covers-from-basic-balls-are-effective}
        A cover of $X$  whose elements are finite unions of elements of the form $\overline{B}\cap X$, where $B$ are basic balls in $\mathcal{X}$, is effective.
    \end{proposition}
    \begin{proof}
        This follows from the fact that $X$ is effectively closed, and that the topological closure of a basic ball in $\mathcal{X}$ is effectively closed (see \Cref{ex:closed rational ball}). 
    \end{proof}

   The following construction provides a uniform sequence of refined effective covers at any desired resolution. 
   
	\begin{lemma}\label{lem:lemma-cubiertas-efectivas}
        There is an algorithm which on input $n\in\N$ computes an effective cover $\mathcal{P}_n$ of $X$ with diameter at most  $2^{-n}$, and such that for every $n\in \N$, $\mathcal P_ {n+1}$ refines $\mathcal{P}_n$ and the inclusion of elements in $\mathcal P_ {n+1}$ to elements in $\mathcal P _n$ is decidable uniformly on $n$. Moreover, if $X$ is a zero dimensional set, we can add the extra condition that $\mathcal P$ is a partition of $X$.
	\end{lemma}
	\begin{proof} Using the recursive compactness of $X$, one can compute a cover $\mathcal{P}_0$ of $X$ made by the intersection of $X$ with (the closure of) basic balls of radius one. The algorithm now proceeds inductively on $n \in \NN$. After it has computed $\mathcal P_0,\dots,\mathcal P_n$ with the desired properties, it computes $\mathcal P_{n+1}$ as follows. Let $\{V_i\}_{i \in \NN}$ be a computable enumeration of all finite unions of basic balls. First observe that there is a finite subset $I$ of $\N$ such that the set  $\{V_i : i\in I\}$ satisfies the following conditions:
        \begin{enumerate}
            \item $V_i$ has diameter at most $2^{-(n+1)}$. 
            \item $\overline{V}_i$ is contained in the interior of some of the elements of $\mathcal P_n$
            \item The union of $V_i$ for $i\in I$ contains $X$
            \item If $X$ is zero dimensional, then the sets $\overline{V_i}\cap X$ are disjoint.
        \end{enumerate}
        The existence of such a finite set of basic balls satisfying the first three conditions follows from the fact that $X$ is compact and that the set $\{V_i : i\in \N\}$ is a basis for the topology. If $X$ is zero dimensional, then the last condition can be added because by \Cref{prop.clopens} we can take the basic sets to be clopen. We now observe that it is semi-decidable whether a finite subset $I$ of $\N$ satisfies the conditions above:
        \begin{enumerate}
        \item Finite unions of (the closure of) basic balls are computable sets, so the first condition is routine,
        \item the semi-decidability of the second follows from the fact that $\overline{V}_i\cap X$ is recursively compact and the interior of the elements in the partition are effectively open,
        \item the semi-decidability of the third condition follows directly from the recursive compactness of $X$,
        \item finally, for the fourth condition it suffices to note that in a recursively compact space it is semi-decidable whether two effectively closed sets have empty intersection (recall that in a recursively compact space we can semi-decide whether an effectively closed subset is empty, the intersection of two effectively closed sets is effectively closed, and this is a computable operation on the descriptions of the sets). 
        \end{enumerate}
        Thus we can compute a set $I$ as desired by an exhaustive search, and finally set $\mathcal {P}_{n+1}=\{V_i\cap X : i\in I\}$.  
	\end{proof}
    \begin{proposition}\label{prop:generating-covers-can-be-taken-effective}
        If $\Gamma\curvearrowright X$ admits a generating open cover, then it admits a generating effective cover.
    \end{proposition}
    \begin{proof}
        Let $\mathcal P$ be a generating open cover of $\Gamma\curvearrowright X$. We construct a new cover $\mathcal P'$ as follows. Let $\mathcal J$ be the collection of all sets of the form $B_i\cap X$ which are contained in some element of $\mathcal P$, where $B_i$ is a basic ball in $\mathcal X$. The fact that basic balls constitute a basis for a topology of $\mathcal X$, together with the compactness of $X$, imply that there is a finite subset $\mathcal J'\subset \mathcal J$ which is an open cover of $X$.  As $\mathcal J'$ refines $\mathcal P$, it follows that it is also a generating cover of $\Gamma\curvearrowright X$.  
        Finally we define $\mathcal P'$ as the closed cover obtained by taking the topological closure of all elements in $\mathcal J'$. By construction, $\mathcal P'$ is a closed cover of $X$ which is generating for the action $\Gamma\curvearrowright X$, and it follows from  \Cref{prop:Covers-from-basic-balls-are-effective} that $\mathcal P'$ is an effective cover. 
    \end{proof}

We remark that in the previous proposition we do not claim the cover $\mathcal P'$ to be computable uniformly from a description of $\Gamma\curvearrowright X$, just its existence. 
    
    \begin{proposition}\label{prop:generating-covers-can-be-taken-effective-and-partition-in-zero-dimension}
        If  $\Gamma\curvearrowright X$ admits a generating open cover and $X$ is zero dimensional, then $\Gamma\curvearrowright X$  admits an effective generating cover which is also a partition. 
    \end{proposition}
    \begin{proof}
        
    It follows from the hypotheses that $X$ admits a generating open cover $\mathcal P$  which is a partition of $X$. Indeed, let $\mathcal P_0$ be an open cover for $X$  which is generating for $\Gamma\curvearrowright X$. As $X$ is zero-dimensional, there is an open cover $\mathcal P$ of $X$ which refines $\mathcal P_0$ and constitutes a partition of $X$. Observe that the cover $\mathcal P$ is also generating for $\Gamma\curvearrowright X$ because it refines $\mathcal P_0$.

  Repeating the construction in the proof of~\Cref{prop:generating-covers-can-be-taken-effective} with the cover $\mathcal P=\{P_0,\dots,P_n\}$, we obtain a finite open cover $\mathcal J'$ of $X$ which refines $\mathcal P$, and such that every element in $\mathcal J'$ has the form $B\cap X$, where $B$ is a basic ball in $\mathcal X$.
    We define a new cover $\mathcal P'=\{P'_0,\dots,P'_n\}$ by the following condition. For every $i \in \{0,\dots,n\}$, $P'_i$ is the union of all elements in $\mathcal J'$ which are contained in $P_i$. The fact that $\mathcal P$  is a partition of $X$ implies that the same holds for $\mathcal P'$. It is also clear that $\mathcal P'$ is a generating cover. We now verify that $\mathcal P'$  is an effective cover. Indeed, the fact that $\mathcal P'$ is a partition implies that the elements of $\mathcal P'$ are closed in $X$. In particular we can write \[P'_i=\overline{P'_i}=\bigcup_{m\in M} \overline{B_m}\cap X\] where $M$ is a finite subset of $\N$ and $B_m$  are rational balls. Here topological closures can be taken in $X$ and in $\mathcal X$. Both topological closures coincide because $X$ is a closed subset of $\mathcal X$ and we consider $X$ with the subspace topology. It follows from the previous set equality and~\Cref{prop:Covers-from-basic-balls-are-effective} that $\mathcal P'$ is also an effective cover. 
        \end{proof}




	\subsection{Proof of main results}

We start with the following observation for computable actions on zero dimensional spaces.  

    \begin{proposition}\label{cor:cantor-representative}
        Let $\Gamma\curvearrowright X$ be a computable action, where $X$ is zero-dimensional subset of a computable metric space. Then  $\Gamma\curvearrowright X$ is topologically conjugate to a computable action $\Gamma\curvearrowright Y$, where $Y$ is an effectively closed subset of $\{\symb{0},\symb{1}\}^\N$. 
    \end{proposition}
    \begin{proof}
        This follows by conjugating the action $\Gamma\curvearrowright X$ with a computable homeomorphism between $X$ and an effectively closed subset of $\{\symb{0},\symb{1}\}^\N$, whose existence is guaranteed by \Cref{thm:compacto-y-cero-dimensional-tiene-copia-en-el-cantor}.
    \end{proof}

   We are now ready to present the proof of our main result. 
    
    \begin{theorem}[\Cref{thm:zero_dim_effective_extension}]\label{prop:main-result}
		Let $\Gamma\curvearrowright X$ be an EDS, where $\Gamma$ is a finitely generated and recursively presented group. Then $\Gamma\curvearrowright X$ is the topological factor of a computable action of an effectively closed subset of $\{\symb{0},\symb{1}\}^{\NN}$. 
	\end{theorem} 
    \begin{proof}
        We start by replacing $\Gamma \curvearrowright X$ by a computable representative, as defined in~\Cref{sec.eds}.  As explained at the beginning of this section, we will construct a topological extension of $\Gamma\curvearrowright X$ as the inverse limit of a sequence of subshifts. 
        
        Let $S$ be a finite set of generators of $\Gamma$ and $F(S)$ be the free group generated by $S$. Let $F(S)\curvearrowright X$ be the action induced from $\Gamma \curvearrowright X$ and notice it is computable. We apply the tools from~\Cref{subsec:subshift-covers} to the action $F(S)\curvearrowright X$. Let $(\mathcal P ^n)_{n\in\N}$ be a uniform sequence of effective covers as in  \Cref{lem:lemma-cubiertas-efectivas}. For each $\mathcal P^n=\{P^n_0,\dots,P^n _{m_n}\}$ we set $A_n$ to be the alphabet $\{0,\dots,m_n\}$, and we let $Y(F(S)\curvearrowright X, \mathcal P ^n)$ be the subshift cover associated to the action $F(S)\curvearrowright X$ and the cover $\mathcal P ^n$. By \Cref{prop:effective-covers-give-effective-subshifts}, each $Y(F(S)\curvearrowright X, \mathcal P ^n)$ is an effective subshift, and the sequence $(Y(F(S)\curvearrowright X, \mathcal P ^n))_{n\in\N}$ is uniformly effective.

        It is possible that some elements in $F(S)$ which are mapped to the identity element in $\Gamma$ (and thus act trivially on $X$), act non-trivially on $Y(F(S)\curvearrowright X, \mathcal P ^n)$. Indeed, if $w\in F(S)$ is mapped to the identity element in $\Gamma$, and there are distinct $i,j$ such that $P^n_{i} \cap P^n_{j} \neq \varnothing$, we could associate distinct elements to $1_{F(S)}$ and $w$, and thus $w$ acts non-trivially on $Y(F(S)\curvearrowright X, \mathcal P ^n)$.
        This is solved as follows. For each $n$, we define $Y_n$ by  
        \[Y_n = Y(F(S)\curvearrowright X, \mathcal P ^n)\cap \widehat{A_n^\Gamma}\subset A_n^{F(S)},\]
        where $\widehat{A_n^\Gamma}$ is the pullback of $A_n^\Gamma$ to $F(S)$. As $\Gamma$ is recursively presented, the set $\widehat{A_n^\Gamma}$ is effectively closed (\Cref{example:fullshift_RP_is_effective}). As $Y_n$ is intersection of two effectively closed sets, then $Y_n$ is effectively closed. Furthermore, as $Y(F(S)\curvearrowright X, \mathcal P ^n)$ is uniform and the second set in the intersection ($\widehat{A_n^\Gamma}$) is the same for each $n$ it follows that $(Y_n)_{n \in \NN}$ is uniformly effectively closed. 

        We now build a zero-dimensional topological extension of $\Gamma \curvearrowright X$, and then we verify the computability of the construction. We define a $F(S)$-subshift $Z_n$ on alphabet $B_n=A_0\times\dots\times A_n$ by the following two conditions:
        \begin{enumerate}
            \item $Z_n \subset Y_0 \times \dots \times Y_n$. 
            \item For each $z\in Z_n$ and for each $w\in F(S)$ with $z(w)=(e_0,\dots,e_n)$ we have \[P^n_{e_n}\subset\dots\subset P^0_{e_0}.\]
        \end{enumerate}


        For each $n$, let $\pi_{n+1}\colon Z_{n+1}\to Z_n$ be the map which removes the last component of the tuple. Then let $Z$ be the inverse limit of dynamical systems 
                  \[Z=\varprojlim Z_n=\{(z_n)\in\prod_{n \in \NN} Z_n\mid\pi_{n+1}(x_{n+1})=x_n, n\geq 1\}\subset\prod_{n \in \NN} B_n^{F(S)}.\]

        We now go into some computability considerations.  First observe that for each $n\in \NN$, $Z_n$ is an effectively closed subset of $B_n^{F(S)}$. As $(Y_n)$ is a uniform sequence of effectively closed sets, we can semi-decide whether the first condition fails. Moreover the second defining condition of $Z_n$ is decidable, uniformly on $n$, by construction of the sequence $(\mathcal P^n)_{n\in\N}$ (\Cref{lem:lemma-cubiertas-efectivas}). Thus $(Z_n)_{n \in \NN}$ is a uniform sequence of effective $F(S)$-subshifts. It is also clear that $(\pi_n)_{n\geq 1}$ is a uniform sequence of computable functions, which are also factor maps. Now we claim that $Z$ is a recursively compact subset of $\prod_{n \in \NN} B_n^{F(S)}$. This follows from three facts:
        \begin{enumerate}
            \item $Z$ is an effectively closed subset of $\prod_{n \in \NN} B_n^{F(S)}$ by \Cref{prop:inverse-limit-of-EDS}.
            \item The product $\prod_{n \in \NN} B_n^{F(S)}$ is recursively compact  by \Cref{thm:Tychonoff_computable}. 
            \item An effectively closed subset of a recursively compact computable metric space is recursively compact (\Cref{prop.equivalence.comp.closed.subset}).
        \end{enumerate}
        The action $F(S)\curvearrowright \prod_{n \in \NN} B_n^{F(S)}$ is computable by \Cref{prop:countable-product-of-EDS}, and then the same holds for $F(S)\curvearrowright Z$. By construction, every element $w\in F(S)$ which corresponds to the identity element in $\Gamma$  acts trivially on $Z$, and thus we can define an action $\Gamma\curvearrowright Z$ by the expression $gz=wz$, where $w\in F(S)$ is any element with $\underline w = g$. It remains to verify that $\Gamma\curvearrowright Z$ is a topological extension of $\Gamma\curvearrowright X$. This is well known, but we sketch the argument for completeness. 
        
        We define a topological factor $\phi\colon Z\to X$ as follows. First, for each $n\in\N$ we define $\nu_n\colon Z\to Y_n$ as the function which sends $z=(z_n)_{n\in\N}$ to the configuration $\nu_n(z)$ given by $\nu_n(z)(w) = (z_n)(w)(n)$, the $n$-th component of $(z_n)(w)\in B_n$. We now define $\phi(z)$ by the expression
            \[\{\phi(z)\}=\bigcap_{n\in\N} P^n_{\nu_n(z)(1_{F(S)})}.\]

        The function $\phi$ is well defined, as a nested intersection of nonempty compact sets with diameter tending to zero must be a singleton. It is straightforward to verify that $\phi$ is continuous, surjective, and it commutes with the actions $\Gamma\curvearrowright X$ and $\Gamma\curvearrowright Z$.

        We have proved that $\Gamma\curvearrowright Z$ is a computable action on a recursively compact and zero-dimensional subset of a computable metric space, and that it is a topological extension of $\Gamma\curvearrowright X$. Now by   \Cref{cor:cantor-representative}, the action $\Gamma\curvearrowright Z$ is topologically conjugate to a computable action on an effectively closed subset of $\{\symb{0},\symb{1}\}^\N$. This finishes the proof. \end{proof}
We end this section by proving that when the action admits a generating partition, the extension can be taken as an effective subshift. In the zero-dimensional regime this is the computable version of the fact that subshifts are precisely the expansive actions on zero-dimensional sets. 
 
    \begin{theorem}\label{prop:expansive-actions-are-factors-of-effective-subshifts}
		Let $\Gamma\curvearrowright X$ be an EDS, where $\Gamma$ is a finitely generated and recursively presented group and the action admits a generating open cover. Then $\Gamma\curvearrowright X$ is a topological factor of an effective $\Gamma$-subshift. 

        Moreover, if $X$ is zero-dimensional, then $\Gamma \curvearrowright X$ is topologically conjugate to an effective $\Gamma$-subshift and the recursively presented hypothesis on $\Gamma$ can be removed. 
	\end{theorem}
    \begin{proof}
        We first prove the result assuming that $\Gamma$ is recursively presented. By \Cref{prop:generating-covers-can-be-taken-effective}, $\Gamma\curvearrowright X$ admits an effective generating cover $\mathcal P$. Now let $S$ be a finite set of generators of $\Gamma$ and let $F(S)$ be the free group generated by $S$. Let $F(S)\curvearrowright X$ be the action induced by $\Gamma\curvearrowright X$, and consider the subshift $Y(F(S) \curvearrowright X, \mathcal P)$ defined in~\Cref{subsec:subshift-covers}. By \Cref{prop:effective-covers-give-effective-subshifts} we have that $Y(F(S) \curvearrowright X, \mathcal P)$ is effective. 
        
        Observe that some elements in $F(S)$ which are mapped to the identity element in $\Gamma$ and thus act trivially on $X$, have a non-trivial action on $Y(F(S) \curvearrowright X, \mathcal P)$. In order to solve this, we define $Y$ by  
        \[Y = Y(F(S) \curvearrowright X, \mathcal P)\cap \widehat{A^\Gamma},\]
        
        where $\widehat{A^\Gamma}$ is the pullback of $A^\Gamma$ in $F(S)$. As $\Gamma$ is recursively presented, the set $\widehat{A^\Gamma}$ is effectively closed (\Cref{example:fullshift_RP_is_effective}), and thus $Y$ is an effective subshift. 
        
        It is a standard fact that $X$ is a topological factor of $Y$ thanks to the fact that $\mathcal P$ is a generating cover. We sketch the argument for completeness.  Let $(F_n)_{n\in\mathbb{N}}$ be an increasing sequence of finite subsets of $F(S)$ whose union is $F(S)$. For each $y\in Y(F(S) \curvearrowright X, \mathcal P)$, we define $D(y)$ as the following subset of $X$: 
		\[ D(y) = \bigcap_{n\in\N } D(y|_{F_n}). \]
        Recall that $D(p)$, for a pattern $p$, was defined in \Cref{subsec:subshift-covers}. For each $y\in Y$ the set $D(y)$ is a singleton, being a nested intersection of nonempty compact sets with diameter tending to zero. Indeed, for each $n$ the set $D(y|_{F_n})$ is contained in the cover $\bigvee _{g\in F_n} g^{-1}\mathcal P$, and the diameters of these elements tend to zero by definition of generating cover. 

        We can now define a topological factor map $\phi\colon Y\to X$   by the expression
		\[\{\phi(y)\}=D(y).\]

        It is straightforward to verify that $\phi$ is continuous, surjective, and it commutes with the actions $\Gamma\curvearrowright X$ and $\Gamma\curvearrowright Y$.  

        We have proved the claim for recursively presented groups. We now prove that if $X$ is zero-dimensional, the hypothesis that $\Gamma$ is recursively presented can be removed. Note that if $X$ is zero-dimensional, then by \Cref{prop:generating-covers-can-be-taken-effective-and-partition-in-zero-dimension} we can assume that $\mathcal P$ is also a partition, which makes $\phi$ injective, and therefore a topological conjugacy. In particular, it preserves the set of group  elements which act trivially. More precisely, recall that $F(S)\curvearrowright X$ is the action induced by $\Gamma\curvearrowright X$. As $F(S)$ has decidable word problem, the previous construction shows that the action $F(S)\curvearrowright X$ is topologically conjugate to an effectively closed $F(S)$-subshift $Y$. Being a topological conjugacy, every element in $F(S)$ which maps to the trivial element of $\Gamma$ must act trivially on $Y$. Thus we can define a $\Gamma$-subshift $Z$ given by \[Z= \{z\in A^\Gamma: \mbox{ there exists } y \in Y \mbox{ where } y(w) = z(\underline{w}) \mbox{ for every } w \in F(S) \}.\] 
        It is straightforward that the pullback $\widehat{Z}$ to $F(S)$ is precisely $Y$, and thus $Z$ is an effective $\Gamma$-subshift as required. \end{proof}
 
	\section{Simulation by subshifts of finite type}\label{sec:simu}

    The main motivation behind~\Cref{thm:zero_dim_effective_extension} is that it can be used to enhance several results in the literature that characterize zero-dimensional effective dynamical systems as topological factors of subshifts of finite type. We shall provide a general definition that will help us to summarize most of these results.

    Consider two groups $\Gamma,H$ and suppose there exists an epimorphism $\psi\colon \Gamma \to H$, that is, that $H$ is a quotient of $\Gamma$. Given an action $H \curvearrowright X$, recall that the \define{induced action}  of $H \curvearrowright X$ to $\Gamma$ is the action $\Gamma \curvearrowright X$ which satisfies that $g\cdot x = \psi(g)\cdot x$ for every $g \in \Gamma$.

    \begin{definition} Let $\Gamma$ and $H$ be finitely generated groups and $\psi\colon \Gamma \to H$ be an epimorphism. We say $\Gamma$ \define{simulates} $H$ (through $\psi$) if for every effectively closed set $X\subset \{\symb{0},\symb{1}\}^{\NN}$ and every computable action $H \curvearrowright X$ there exists a $\Gamma$-SFT $Y$ such that the induced action $\Gamma \curvearrowright X$ is a topological factor of $\Gamma \curvearrowright Y$.
	\end{definition}

    If the map $\psi$ in the definition above is clear then we don't mention it (for instance, if $\Gamma$ is a direct product of $H$ with another group we assume that $\psi$ is the projection to $H$). Furthermore, if $\psi$ is an isomorphism, we say that $\Gamma$ is a \define{self-simulable} group. Let us briefly summarize a few of the simulation results available in the literature

    \begin{theorem}\label{thm:simulation_results}
        The following simulation results hold\footnote{In references~\cite{Hochman2009b,BS2018,Barbieri_2019_DA} the simulation results are stated differently, using subactions rather than an induced action. However, they also hold in our terminology (the result in~\cite{Barbieri_2019_DA} is stated only for expansive actions, but it can be generalized with a slight modification in the proof which will be covered in an upcoming article).}
		\begin{enumerate}
			\item \cite{Hochman2009b} $\Gamma = \ZZ^3$ simulates $\ZZ$.
			\item \cite{BS2018} Let $H$ be a finitely generated group, $d \geq 2$ and $\varphi \colon G \to \operatorname{GL}_d(\ZZ)$. Then the semidirect product $\Gamma = \ZZ^d \rtimes_{\varphi} H$ simulates $H$ (through the projection epimorphism).
			\item \cite{Barbieri_2019_DA} Let $H,G_1,G_2$ be finitely generated groups. Then $\Gamma = H \times G_1 \times G_2$ simulates $H$.
            \item \cite{barbieri2023soficity} Let $H,K$ be finitely generated groups, suppose $H$ is nonamenable and that $K$ has decidable word problem, then $\Gamma = H \times K$ simulates $H$.
            \item \cite{Barbieri_Sablik_Salo_2021} The following groups are self-simulable:
            \begin{itemize}
                \item The direct product of any two finitely generated nonamenable groups.
                \item The general linear group $\operatorname{GL}_n(\ZZ)$ and the special linear group $\operatorname{SL}_n(\ZZ)$ for $n \geq 5$.
        	    \item Thompson's group $V$ and the Brin-Thompson's groups $nV$ for $n \geq 1$.
        	    \item Finitely generated non-amenable Branch groups.
        	    \item The automorphism group $\operatorname{Aut}(F_n)$ and the outer automorphism group $\operatorname{Out}(F_n)$ of the free group $F_n$ on $n$ generators for $n \geq 5$.
        	    \item The braid groups $B_n$ for $n \geq 7$ strands.
            \end{itemize}
		\end{enumerate}
    \end{theorem}

    We can enhance all of the above results using~\Cref{thm:zero_dim_effective_extension}. 

    \begin{theorem}[\Cref{thm:simulation-by-SFTs-enhanced-version-GOTY}]
        Let $\Gamma,H$ be finitely generated groups and $\psi\colon \Gamma \to H$ be an epimorphism. Suppose that $H$ is recursively presented, then $\Gamma$ simulates $H$ if and only if for every effective dynamical system $H \curvearrowright X$ the induced action of $\Gamma$ is a topological factor of a $\Gamma$-SFT.
\end{theorem}

    \begin{proof}
        Suppose first that for every EDS $H \curvearrowright X$ the induced action of $\Gamma$ is a topological factor of a $\Gamma$-SFT. As a computable action on an effectively closed zero-dimensional set is in particular an EDS, it follows that $\Gamma$ simulates $H$. 
        
        Conversely, let $H \curvearrowright X$ be an EDS. As $H$ is recursively presented,~\Cref{thm:zero_dim_effective_extension} implies that there exists an effectively closed set $\widetilde{X}\subset \{\symb{0},\symb{1}\}^{\NN}$ and a computable action $H \curvearrowright \widetilde{X}$ such that $H \curvearrowright X$ is a topological factor of $H \curvearrowright \widetilde{X}$. Taking the induced action of $\Gamma$ of these two actions it follows that $\Gamma \curvearrowright X$ is a topological factor of $\Gamma \curvearrowright\widetilde{X}$. As $\Gamma$ simulates $H$, there exists a $\Gamma$-SFT $Y$ which factors onto $\Gamma \curvearrowright \widetilde{X}$ and thus onto $\Gamma \curvearrowright X$. 
    \end{proof}

    Let us illustrate how~\Cref{thm:simulation-by-SFTs-enhanced-version-GOTY} can be applied with a few examples.
    
	\begin{example}
		 As mentioned in~\Cref{example:torus_matrix_multiplication}, the action $\operatorname{GL}_n(\ZZ) \curvearrowright \RR^n/\ZZ^n$ by left matrix multiplication is an EDS. For $n \geq 5$ the group $\operatorname{GL}_n(\ZZ)$ is both recursively presented and self-simulable, thus it follows by~\Cref{thm:simulation-by-SFTs-enhanced-version-GOTY} that there exists a $\operatorname{GL}_n(\ZZ)$-SFT which topologically factors onto $\operatorname{GL}_n(\ZZ) \curvearrowright \RR^n/\ZZ^n$. 
	\end{example}

    \begin{example}
        Let $\{\beta_i\}_{i=1}^4$ be four computable points in $\RR/\ZZ$. Consider $F_2 \times F_2 = \langle a_1,a_2 \rangle \times \langle a_3,a_4 \rangle$ and the action $F_2 \times F_2 \curvearrowright \RR/\ZZ$ given by $a_i\cdot x = x+\beta_i \bmod{\ZZ}$ for $i \in \{1,\dots,4\}$ and remark that this action is in fact the pullback of an action $\ZZ^4 \curvearrowright \RR/\ZZ$.  As $F_2\times F_2$ is recursively presented and self-simulable, it follows by~\Cref{thm:simulation-by-SFTs-enhanced-version-GOTY} that there exists an $(F_2\times F_2)$-SFT which topologically factors onto $F_2 \times F_2 \curvearrowright \RR/\ZZ$. 
    \end{example}

	\begin{example} Let $G$ be a group, $n \geq 2$ and consider $\Delta_n(G) \subset G^n$ the space of all $n$-tuples of elements of $G$ whose product is trivial, that is \[ \Delta_n(G) = \{ (x_1,\dots,x_n)\in G^n : x_1\cdots x_n = 1_{G}\}\]    
	The braid group $B_n$ on $n$ strands acts on $\Delta_n(G)$ by \[\sigma_i (x_1,\dots,x_{i-1},x_{i},x_{i+1},\dots,x_n) = (x_1,\dots,x_{i-1},x_{i+1}, x_{i+1}^{-1}x_ix_{i+1},x_{i+2},\dots,x_n). \]  
	Where $1 \leq i \leq n-1$ and $\sigma_i$ is the element of $B_n$ which crosses strands $i$ and $i+1$ (for the definition of braid group, refer to~\cite{Birman_Brendle_2005_braidgroupssurvey}).	
	If we let $G$ be any compact topological group which admits a computable structure such that group multiplication is a computable map, for instance the group of unitary complex matrices $G = U(k) \leqslant \operatorname{GL}_k(\CC)$ for some $k \geq 1$, then $\Delta_n(G)$ is an effectively closed set of $G^n$ and the action $B_n \curvearrowright \Delta_n(G)$ is computable. It follows that if $n \geq 7$ then there exists a $B_n$-SFT which topologically factors onto $B_n \curvearrowright \Delta_n(G)$.
	\end{example}

  In~\cite{Barbieri_Sablik_Salo_2021} it was proven that Thompson's group $F$ (see~\Cref{ex:thompson}) is self-simulable if and only if $F$ is non-amenable, which is a longstanding open question. We can therefore strengthen the characterization of the potential amenability of $F$ in the following way:

    \begin{corollary}
        Thompson's group $F$ is amenable if and only if there exists an EDS $F \curvearrowright X$ which is not the topological factor of an $F$-SFT.
    \end{corollary}



    In the literature, there are a few simulation results which apply exclusively to expansive maps. More precisely, let $\Gamma$ and $H$ be finitely generated groups and $\psi\colon \Gamma \to H$ be an epimorphism. We say $\Gamma$ \define{simulates expansive actions of} $H$ (through $\psi$) if for every effective subshift $X\subset A^{H}$ there exists a $\Gamma$-SFT $Y$ such that the pullback subshift $\Gamma \curvearrowright X$ is a topological factor of $\Gamma \curvearrowright Y$. The most famous results are due to Aubrun and Sablik~\cite{AubrunSablik2010} and Durand, Romaschenko and Shen~\cite{DurandRomashchenkoShen2010} which state that $\ZZ^2$ simulates expansive actions of $\ZZ$ (through the projection epimorphism).

    In light of~\Cref{prop:expansive-actions-are-factors-of-effective-subshifts}, we can improve the expansive simulation theorems in the following way

    \begin{theorem}
        Let $\Gamma,H$ be finitely generated groups and $\psi\colon \Gamma \to H$ be an epimorphism. Suppose that $H$ is recursively presented, then $\Gamma$ simulates expansive actions of $H$ if and only if for every effective dynamical system $H \curvearrowright X$ which admits a generating cover (in particular, every expansive EDS) the induced action of $\Gamma$ is a topological factor of a $\Gamma$-SFT.
    \end{theorem}

 \section{Algebraic actions as effective dynamical systems}\label{sec:algebraic actions}

In this section we will show that, under mild conditions, algebraic actions are effective dynamical systems and thus we can apply the results of the previous sections on them.

\begin{definition}
    Let $\Gamma$ be a countable group and $X$ be a compact metrizable abelian group. We say that $\Gamma \curvearrowright X$ is an \define{algebraic action} if $\Gamma$ acts by continuous automorphisms of $X$.
\end{definition}

Next we provide a few examples to give an idea of what algebraic actions look like. The reader can find a plethora of interesting examples in~\cite{Schmidt1995}.

\begin{example}
    Let $X = (\RR/\ZZ)^2$ be the torus (with pointwise addition). For every $A \in \operatorname{GL}_2(\ZZ)$ the action $\ZZ \curvearrowright X$ given by $n\cdot x = A^{n}x$ is algebraic.
\end{example}

\begin{example}\label{example:algebraic_action}
    Let $\Gamma$ be a countable group and consider $X = (\RR/\ZZ)^{\Gamma}$ with pointwise addition as the group operation. Then $X$ is a compact metrizable abelian group and the shift action $\Gamma \curvearrowright X$ given by $(g\cdot x)(h) = x(g^{-1}h)$ for every $x \in X$ and $g,h \in \Gamma$ is algebraic.
\end{example}

\begin{example}
    Let $e_1 = (1,0), e_2 = (0,1)$ in $\ZZ^2$ and consider \[X = \{ x \in (\RR/\ZZ)^{\ZZ^2} : \mbox{for every } u \in \ZZ^2, x(u)+x(u+e_1)+x(u+e_2) = 0 \bmod{\ZZ}\}.\] 
    Then $X$ is a compact metrizable abelian group (with pointwise addition) and the shift action $\ZZ^2 \curvearrowright X$ is algebraic.
\end{example}

An interesting property of algebraic dynamical systems is that they can all be represented in a canonical way through the use of Pontryagin duality. This identification is well established in the literature (it can be found for $\ZZ^d$ in~\cite{Schmidt1995} and for a general countable group in~\cite[Chapter 13]{KerrLiBook2016}), so we shall provide a relatively short account of the main facts. A reader who is already familiar with the basics of algebraic actions of countable groups can skip directly to the paragraph above the definition of recursively presented algebraic action (\Cref{def:recpresented_algebraic_action}).

Let $\Gamma$ be a countable group. The \define{integral group ring} $\ZZ[\Gamma]$ (usually denoted in the literature as $\ZZ\Gamma$ to simplify the notation) is the set of all finitely supported maps $p \colon \Gamma \to \ZZ$ endowed with pointwise addition and with polynomial multiplication. That is, for $p,q \in \ZZ[\Gamma]$ and $g\in \Gamma$ then $(p+q)(g) = p(g)+q(g)$ and $p\cdot q (g) = \sum_{h \in \Gamma} p(gh^{-1})q(h)$. For this reason, we usually denote elements $p \in \ZZ[\Gamma]$ as finite formal sums $p = \sum_{i \in I} c_ig_i$ with $c_i \in \ZZ, g_i \in \Gamma$. We also endow $\ZZ[\Gamma]$ with the $*$-operation given by $p^* = \sum_{i \in I}c_ig_i^{-1}$. We remark that for $x \in (\RR/\ZZ)^{\Gamma}$ and $p = \sum_{i \in I} c_ig_i \in \ZZ[\Gamma]$ the product $x p^*\in (\RR/\ZZ)^{\Gamma}$ is well defined and given for every $h \in \Gamma$ by \[ (x p^*)(h) = \sum_{i \in I} c_ix(hg_i) \bmod{\ZZ}.  \]

Let us recall that for a locally compact abelian group $X$, its \define{Pontryagin dual} is the space $\widehat{X}$ of all continuous homomorphisms from $X$ to $\RR/\ZZ$. Notice that if we endow $\widehat{X}$ with the topology of uniform converge on compact sets and with pointwise addition, then $\widehat{X}$ is also a locally compact abelian group whose elements are called \define{characters}. We recall two well-known properties, their proofs can be found for instance in~\cite{morris_1977}. First, that $X$ is a compact metrizable abelian group if and only if $\widehat{X}$ is countable and discrete. Second, that there exists an isomorphism (as topological groups, and through the canonical evaluation map) between $X$ and its double dual $\widehat{\widehat{\makebox{$X$}}}$.

Now, there is a natural correspondence between algebraic actions $\Gamma \curvearrowright X$ and actions $\Gamma \curvearrowright \widehat{X}$ by automorphisms by letting $(g\cdot \chi)(x) = \chi(g^{-1}x)$ for every $g \in \Gamma$, $x\in X$ and $\chi \in \widehat{X}$. Moreover, $\widehat{X}$, endowed with its corresponding action, has a natural structure of countable left $\ZZ[\Gamma]$-module where for $p = \sum_{i \in I} c_ig_i \in \ZZ[\Gamma]$ and $\chi \in \widehat{X}$ the operation is given by \[ p \cdot \chi = \sum_{i \in I} c_i (g_i\cdot \chi).   \]

Conversely, given a countable left $\ZZ[\Gamma]$-module $M$ endowed with the discrete topology, there is a natural action by (continuous) automorphisms $\Gamma \curvearrowright M$ given by left multiplication by the monomials $p_g = 1\cdot g$ using the module product operation, and as before, this induces an action by continuous automorphisms in $\widehat{M}$. In particular, this means that there is a one-to-one correspondence (up to isomorphism) between algebraic actions of $\Gamma$ and countable left $\ZZ[\Gamma]$-modules.

An algebraic action $\Gamma \curvearrowright X$ is called \define{finitely generated} if its associated $\ZZ[\Gamma]$-module is finitely generated. In this case we can identify the module up to isomorphism with $\ZZ[\Gamma]^n/J$ for some integer $n \geq 1$ and $J$ a left-submodule of the free module $\ZZ[\Gamma]^n$. Passing through the Pontryagin dual, this means that in this case, without loss of generality, we can identify \[ X = \{ (x_1,\dots,x_n) \in ((\RR/\ZZ)^{\Gamma})^n : \sum_{i=1}^n x_i p_i^* = 0 \mbox{ for every } (p_1,\dots,p_n) \in J\}. \]
With the action $\Gamma\curvearrowright \widehat{\widehat{\makebox{$X$}}}$ being the left shift as in~\Cref{example:algebraic_action}. An algebraic action $\Gamma \curvearrowright X$ is called \define{finitely presented} if the associated $\ZZ[\Gamma]$-module is finitely presented. In this case the module is isomorphic to $\ZZ[\Gamma]^n / \ZZ[\Gamma]^k P$ for some $n \in \NN$ and a $k \times n$ matrix $P \in \mathcal{M}_{k \times n}(\ZZ[\Gamma])$. Therefore in this case we can identify \[ X = \{ (x_1,\dots,x_n) \in ((\RR/\ZZ)^{\Gamma})^n : \sum_{i=1}^n x_i P_{i,j}^* = 0 \mbox{ for every } j \in \{1,\dots,k\}\}. \]

We remark that the three examples given at the start of this section are finitely presented algebraic actions.

In order to state our results in their full generality, we will introduce an intermediate notion for algebraic actions which we call being ``recursively presented''. Let $S$ be a finite set of generators for a group $\Gamma$ and for $w \in S^*$ denote by $\underline{w}$ its corresponding element of $\Gamma$. Let $\ZZ[S^*]$ be the space of finitely supported maps $p \colon S^* \to \ZZ$. To each map $p$ we can associate $\underline{p}\in \ZZ[\Gamma]$ by letting $\underline{p}= \sum_{w \in \supp(p)} p(w)\underline{w}$. 

\begin{definition}\label{def:recpresented_algebraic_action}
    Let $\Gamma$ be a finitely generated group and $S$ a finite set of generators. A finitely generated algebraic action $\Gamma \curvearrowright X$ is called \define{recursively presented} if there exists $n \in \NN$ and a computable total map $g \colon \NN \to \ZZ[S^*]^n$ such that up to isomorphism the action is the shift map on \[  X = \{ (x_1,\dots,x_n) \in ((\RR/\ZZ)^{\Gamma})^n : \sum_{i=1}^n x_i\underline{g(k)_i}^* = 0 \mbox{ for every } k \in \NN\}. \]
\end{definition}

Notice that finitely presented algebraic actions are recursively presented, as they can be represented by the maps $g\colon \NN \to \ZZ[S^*]^n$ which are eventually the zero of $\ZZ[S^*]^n$ and thus computable.

\begin{theorem}\label{theorem:algebraic_actions_are_EDS}
    Let $\Gamma$ be a finitely generated and recursively presented group. Every recursively presented algebraic action $\Gamma \curvearrowright X$ is an EDS.
\end{theorem}

\begin{proof}
    Fix $S$ a finite set of generators of $\Gamma$. Consider $n \in \NN$ and a computable total map $g\colon \NN \to \ZZ[S^*]^n$ which define the recursively presented algebraic action $\Gamma \curvearrowright X$. As the space $(\RR/\ZZ)$ is recursively compact (with the structure given by the rationals and the euclidean metric), it follows by~\Cref{thm:Tychonoff_computable} that the product space $((\RR/\ZZ)^{S^*})^n$ is also recursively compact. For $x \in (\RR/\ZZ)^{S^*}$ and $p \in \ZZ[S^*]$ we formally define $xp^* \in (\RR/\ZZ)^{S^*}$ by $(xp^*)(w) = \sum_{u \in \supp(p)} x(wu)p(u)$. Notice that for $w \in S^*$ we have $(x\underline{p}^{*})(\underline{w}) = (xp^*)(w)$.
    
    Consider the space $Y$ of all $(x_1,\dots,x_n)\in  ((\RR/\ZZ)^{S^*})^n$ which satisfy the following two conditions:
    \begin{enumerate}
        \item For every $u,v \in S^*$ such that $\underline{u}=\underline{v}$, and for every $i \in \{1,\dots,n\}$ we have $x_i(u) = x_i(v)$.
        \item For every $k \in \NN$ we have that $\sum_{i=1}^n x_i g(k)_i^* = 0$.
    \end{enumerate}

    We claim that $Y$ is an effectively closed subset of $((\RR/\ZZ)^{S^*})^n$. Indeed, as $\Gamma$ is recursively presented, the set $\texttt{EQ}(\Gamma) = \{ (u,v)\in S^* \times S^* : \underline{u}=\underline{v}\}$ is recursively enumerable. For words $u,v$ consider the map $f_{u,v} \colon ((\RR/\ZZ)^{S^*})^n \to (\RR/\ZZ)^n$ given by \[f_{u,v}(x_1,\dots,x_n) = (x_1(u)-x_1(v), \dots, x_n(u)-x_n(v)).\]
    It is clear that each $f_{u,v}$ is computable (as projections are computable in the product space), and that the collection $\{f_{u,v}\}_{(u,v) \in \texttt{EQ}(\Gamma)}$ is uniformly computable. As $\{0\}$ is an effectively closed subset of $(\RR/\ZZ)^n$, it follows that $Y_{(1)} = \bigcap_{(u,v) \in \texttt{EQ}(\Gamma)} f_{u,v}^{-1}(\{0\})$ is an effectively closed set. Remark that this is precisely the set of elements of $((\RR/\ZZ)^{S^*})^n$ which satisfy condition (1).

    Given $k \in \NN$, and $v \in S^*$, consider the map $h_{k,v} \colon ((\RR/\ZZ)^{S^*})^n \to \RR/\ZZ$ given by \[ h_{k,v}(x_1,\dots,x_n) =  \left(\sum_{i=1}^n x_i g(k)_i^*\right)(v) = \sum_{i =1 }^n \sum_{u \in \supp(g(k)_i)} x_i(vu) g(k)_i(u).  \]
    It is clear from the fact that $g$ is a total computable function that the collection $(h_{k,v})_{k \in \NN, v \in S^*}$ is uniformly computable. It follows that $Y_{(2)} = \bigcap_{k \in \NN}\bigcap_{v \in S^*} h_{k,v}^{-1}(\{0\})$ is an effectively closed set as well. Remark that this is precisely the set of elements of $((\RR/\ZZ)^{S^*})^n$ which satisfy condition (2).

   As $Y = Y_{(1)}\cap Y_{(2)}$, it follows that $Y$ is an effectively closed subset of $((\RR/\ZZ)^{S^*})^n$.

   Let us define for each $s \in S$ the computable map $t_s \colon Y \to Y$ given by $(t_s(x_1,\dots,x_n))(u) = (x_1,\dots,x_n)(s^{-1}u)$. It is clear by condition (1) that these maps induce an action of $\Gamma$, thus $\Gamma \curvearrowright Y$ is a computable action on an effectively closed set. Finally, it follows from condition (2) that the map $\phi \colon X \to Y$ given by $(\phi(x_1,\dots,x_n))(w) = (x_1,\dots,x_n)(\underline{w})$ is a topological conjugacy and thus $\Gamma \curvearrowright X$ is an EDS.\end{proof}

   Putting together~\Cref{theorem:algebraic_actions_are_EDS} and \Cref{thm:simulation-by-SFTs-enhanced-version-GOTY}, we obtain the following corollary which in turn, implies~\Cref{thm:finitely_presented_algebraic_action} (as finitely presented algebraic actions are recursively presented).

   \begin{corollary}[Theorem C]
       Let $\Gamma$ be a recursively presented self-simulable group. Then every recursively presented algebraic action of $\Gamma$ is the topological factor of some $\Gamma$-subshift of finite type.
   \end{corollary}

 \section{Other computability notions for shift spaces}\label{sec:computability-on-shift-spaces}

    The goal of this section is to compare the computability notions for subshifts we have introduced in this article with other notions of computability for shift spaces that can be found in the literature. We start by showing that a subshift is an EDS if and only if it is effective as per  \Cref{def:effective_subshift_through_free_group}. 
    
\begin{proposition}\label{prop:EDS_and_effective_subshift_are_equivalent}
        Let $\Gamma$ be a finitely generated group. A subshift $X\subset A^{\Gamma}$ is an EDS if and only if it is effective. 
    \end{proposition}

    \begin{proof}
         Let us recall that for a finitely generated group $\Gamma$ a subshift $X\subset A^{\Gamma}$ is effective if its pullback into the full shift $A^{F(S)}$ is an effectively closed set. We already argued in~\Cref{example:effective_subshift_is_EDS} that every effective subshift is an EDS. Let us fix a finite set of generators $S$ of $\Gamma$ and denote by $\widehat{X}\subset A^{F(S)}$ the pullback subshift of $X$ on the free group $F(S)$.
        Suppose that $\Gamma \curvearrowright X$ is topologically conjugate to a computable action on an effectively closed subset $Y$ of a recursively compact metric space. By~\Cref{cor:cantor-representative}, we may assume without loss of generality that $Y$ is an effectively closed subset of $\{0,1\}^{\NN}$ with the canonical computable structure. As $\Gamma \curvearrowright X$ is topologically conjugate to $\Gamma \curvearrowright \widehat{X}$, there is a topological conjugacy $\phi \colon Y \to \widehat{X}$. For $a \in A$, let $[a] = \{ \widehat{x} \in \widehat{X} : \widehat{x}(1_{F(S)}) = a\}$, for a finite $W\subset F(S)$ and $p \in A^{W}$ we let \[ [p] = \bigcap_{w \in W} w[p(w)] = \{  \widehat{x} \in \widehat{X} : \widehat{x}(w) = p(w) \mbox{ for every } w \in W\}.     \] As $\phi$ is a homeomorphism, we have that $C_a = \phi^{-1}([a])$ is a finite union of cylinder sets in $Y$ and thus an effectively closed subset of $Y$.
    
    We claim that there exists an algorithm that given a finite $W\subset F(S)$ and a pattern $p \colon W \to A$, accepts if and only if $[p]\cap \widehat{X}=\varnothing$. Indeed, given such a $p$, we may compute a description of $C_p = \bigcap_{ w \in W} w \cdot C_{p(w)}$. As the action $\Gamma \curvearrowright Y$ is computable, it follows that each $w\cdot C_{p(w)}$ is an effectively closed set and thus, as the intersection of finitely many effectively closed sets is effectively closed, one may semi-decide whether $C_p = \varnothing$, which occurs precisely when $[p] \cap \widehat{X} =\varnothing$. It follows that the collection $\mathcal{F}_{\mathrm{max}} = \{ p \mbox{ is a pattern } : [p] \cap \widehat{X} = \varnothing\}$ is recursively enumerable. As $\mathcal{F}_{\mathrm{max}}$ is the maximal set of patterns which defines $\widehat{X}$, it follows that $\widehat{X}$ is an effectively closed set.
    \end{proof}
     
    

In the literature, there is also an intrinsic notion of computability for subshifts based on codings of forbidden patterns. This notion was originally defined in~\cite{ABS2017}.

Let $S$ be a finite set of generators of a group $\Gamma$ and $A$ be an alphabet. A partial function $c\colon S^* \to A$ with finite support is called a \define{pattern coding} and its domain is denoted $\operatorname{supp}(c)$. Recall that for $w \in S^*$ we denote by $\underline{w}$ the corresponding element of $\Gamma$. The cylinder set induced by a pattern coding $c$ is given by \[ [c] = \{ x \in A^{\Gamma} : x(\underline{w}) = c(w) \mbox{ for every } w \in \operatorname{supp}(c) \}.      \]

A pattern coding is \define{inconsistent} if there are $u,v \in \operatorname{supp}(c)$ such that $\underline{u}=\underline{v}$ but $c(u)\neq c(v)$. Notice that for inconsistent pattern codings we have $[c] = \varnothing$. The set of all pattern codings is a named set, and thus we can speak about computability of sets of pattern codings. Given a set $\mathcal{C}$ of pattern codings, we can define a subshift $X_{\mathcal{C}}\subset A^{\Gamma}$ by  \[ X_{\mathcal{C}} = A^{\Gamma} \setminus \bigcup_{g \in \Gamma, c \in \mathcal{C}} g[c].  \]

\begin{definition}\label{def:ECP}
    We say that a subshift $X \subset A^{\Gamma}$ is \define{effectively closed by patterns} (ECP) if there exists a recursively enumerable set of pattern codings $\mathcal{C}$ such that $X = X_{\mathcal{C}}$.
\end{definition}

\begin{example}
    For any finitely generated group $\Gamma$, every subshift of finite type $X\subset A^{\Gamma}$ is ECP, as it is given by a finite set of forbidden pattern codings.
\end{example}

In what follows we will show that for general finitely generated groups, being an ECP subshift is strictly weaker than being an EDS (or equivalently, an effective subshift), but that all notions coincide for recursively presented groups. For a subshift $X\subset A^{\Gamma}$ denote by $\mathcal{C}_{\mbox{max}}(X)$ the set of all pattern codings $c$ such that $[c]\cap X = \varnothing$.

\begin{proposition}\label{prop_effective_implies_max_forbidden_is_re}
    Let $\Gamma$ be a finitely generated group. A subshift $X\subset A^{\Gamma}$ is effective if and only if the set $\mathcal{C}_{\mbox{max}}(X)$ is recursively enumerable.
\end{proposition}

\begin{proof}
    Let $\mathcal C_{\mbox{max}}'$ be the set of all pattern codings $c$ such that $[c]\cap X = \varnothing$ and $c$ is consistent on the free group $F(S)$. As the word problem in $F(S)$ is decidable, it follows that  $\mathcal C_{\mbox{max}}(X)$ is recursively enumerable if and only if $\mathcal C_{\mbox{max}}'$ is recursively enumerable. 
    
    Let us recall that $X$ is an effective subshift when $\widehat{X}\subset A^{F(S)}$ is effectively closed. We claim that $\mathcal C_{\mbox{max}}'$ is recursively enumerable if and only if $\widehat{X}\subset A^{F(S)}$ is effectively closed. This is just by definition of effectively closed set, as $\mathcal C_{\mbox{max}}'$ is exactly the set of pattern codings of patterns in the free group whose associated cylinder in $A^{F(S)}$ do not intersect $\widehat X$. \end{proof}

    In particular, as the set $\mathcal{C}_{\mbox{max}}(X)$ defines the subshift $X$, we obtain that every effective subshift is ECP.

    \begin{corollary}
        Let $\Gamma$ be a finitely generated group and $X\subset A^\Gamma$ be an effective subshift. Then $X$ is effectively closed by patterns.
    \end{corollary}

    We shall see that the converse only holds in recursively presented groups. For this we shall need the following lemma. 
    
\begin{lemma}{\cite[Lemma 2.3]{ABS2017}}
    Let $\Gamma$ be a finitely generated and recursively presented group. For every ECP subshift $X\subset A^{\Gamma}$, the set $\mathcal{C}_{\mbox{max}}(X)$ is recursively enumerable.
\end{lemma}

    \begin{corollary}\label{cor:ECP_plus_RP_implies_EDS}
        On recursively presented groups, every ECP subshift is effective. Thus on recursively presented groups the three computability notions are equivalent. 
    \end{corollary}

    We remark that for an effective cover $\mathcal{P}$ of a recursively compact metric space $X$, the subshift $Y(\Gamma \curvearrowright X, \mathcal{P})$ from~\Cref{prop:effective-covers-give-effective-subshifts} is ECP but not necessarily effective. This is fundamentally the reason why we need to ask that $\Gamma$ is recursively presented in~\Cref{thm:zero_dim_effective_extension}.

    A useful consequence of \Cref{cor:ECP_plus_RP_implies_EDS} is that on recursively presented groups, SFTs are  effective subshifts. This follows from the fact that SFTs are ECP. 
\begin{corollary}\label{cor:on recursively presented groups SFTs are effective}
    Let $\Gamma$ be a recursively presented group. Then every $\Gamma$-SFT is effective.
\end{corollary}

Let us show that for non recursively presented groups there are always ECP subshifts which are not effective. Let $A$ be a finite alphabet with $|A|\geq 2$ and $\Gamma$ a finitely generated group. Notice that the full $\Gamma$-shift $A^{\Gamma}$ is ECP as it is defined through the empty set of pattern codings.

\begin{proposition}
    The full $\Gamma$-shift is effective if and only if $\Gamma$ is recursively presented.
\end{proposition}

\begin{proof}
    If $\Gamma$ is recursively presented then the full $\Gamma$-shift is effective by~\Cref{cor:on recursively presented groups SFTs are effective}. Conversely, if the full $\Gamma$-shift is effective then $\widehat{A^{\Gamma}}$ is an effectively closed set and thus by~\Cref{prop_effective_implies_max_forbidden_is_re} the maximal set of forbidden pattern codings is recursively enumerable.
    
    Let $S$ be a finite generating set of $\Gamma$ and consider the pullback $\widehat{A^{\Gamma}}\subset A^{F(S)}$. Fix two distinct symbols $a,b \in A$. Given $w \in S^*$, there is an algorithm that constructs its corresponding element $\underline{w}\in F(S)$ and the pattern $p_w \in A^{\{1_{F(S)},\underline{w}\}}$ with $p_w(1_{F(S)}) = a$ and $p_w(\underline{w})=b$. Clearly $[p_w]\cap \widehat{A^{\Gamma}}=\varnothing$ if and only if $w$ represents the identity in $\Gamma$ and furthermore $\widehat{A^{\Gamma}} = A^{F(S)}\setminus \bigcup_{g \in \Gamma, w \in \operatorname{\texttt{WP}(\Gamma)}} g[p_w]$.
    
     It follows that given $w \in S^*$ we can recursively enumerate the $w$ for which $[p_w]\cap \widehat{A^{\Gamma}}=\varnothing$ which are precisely the elements of $\texttt{WP}(\Gamma)$. Therefore $\Gamma$ is recursively presented.
\end{proof}
    
\section{Topological factors of effective dynamical systems}\label{sec:factors}

We begin by showing that the class of EDS is not closed by topological factors. In particular, it shows that subshifts of finite type on certain groups may have factors that are not EDS. However, we shall show that under certain conditions such as having topological zero dimension, we can still provide computational restrictions that said factors must satisfy.

Given an action $\Gamma \curvearrowright X$, we define its set of periods as \[ \mathtt{Per}(\Gamma \curvearrowright X) = \{ g \in \Gamma : \mbox{ there is } x \in X \mbox{ such that } gx = x\}\]  

The following lemma holds in general for any finitely generated group with decidable word problem. However, we shall only need it for the special case of $\ZZ$-actions.

\begin{lemma}
    Let $\ZZ \curvearrowright X$ be an EDS. Then $\mathtt{Per}(\ZZ\curvearrowright X)$ is a co-recursively enumerable set.
\end{lemma}

\begin{proof}
    We will prove that the complement of $\mathtt{Per}(\ZZ \curvearrowright X)$ is a recursively enumerable subset of $\ZZ$. Without loss of generality, replace $\ZZ \curvearrowright X$ by a computable representative. Then the product space $X^2$ is recursively compact and thus the diagonal $\Delta^2 = \{ (x,x) \in X^2 : x \in X\}$ is an effectively closed subset and thus also recursively compact. As $\ZZ\curvearrowright X$ is computable, it follows that the collection of diagonal maps $f_n\colon X^2 \to X^2$ given by $(x,y) \mapsto (T^n(x),y)$ for $n \in \ZZ$ is uniformly computable, and thus the collection of sets $Y_n = f_n^{-1}(\Delta^2)$ is uniformly recursively compact. It follows that we can recursively enumerate the integers $n$ such that $Y_n \cap \Delta^2 = \varnothing$. In other words, the set \[ \ZZ \setminus \mathtt{Per}(\ZZ \curvearrowright X) = \{ n \in \ZZ : Y_n \cap \Delta^2 = \varnothing\}, \]
    is recursively enumerable.
\end{proof}

The set of periods is an invariant of topological conjugacy, therefore any action $\ZZ\curvearrowright X$ for which the set of periods is not co-recursively enumerable cannot be topologically conjugate to an EDS.

\begin{proposition}
    The class of EDS is not closed under topological factor maps.
\end{proposition}

\begin{proof}
    Let us consider \[X = \{0\} \cup \left\{ z \in \CC : |z| = \frac{1}{n} \mbox{ for some integer }n \geq 1 \right\}. \]
And let $T \colon X \to X$ be given by $T(z) = z\exp\left(2\pi i |z|\right)$.
It is clear that $X$ is a recursively compact subset of $\CC$ with the standard topology and that $T$ is a computable homeomorphism which thus induces a computable action $\ZZ \curvearrowright X$. The map $T$ consists of rational rotations on each circle where the period is given by the inverse of the radius, and thus it is easy to see that $\mathtt{Per}(\ZZ\curvearrowright X) = \ZZ$.

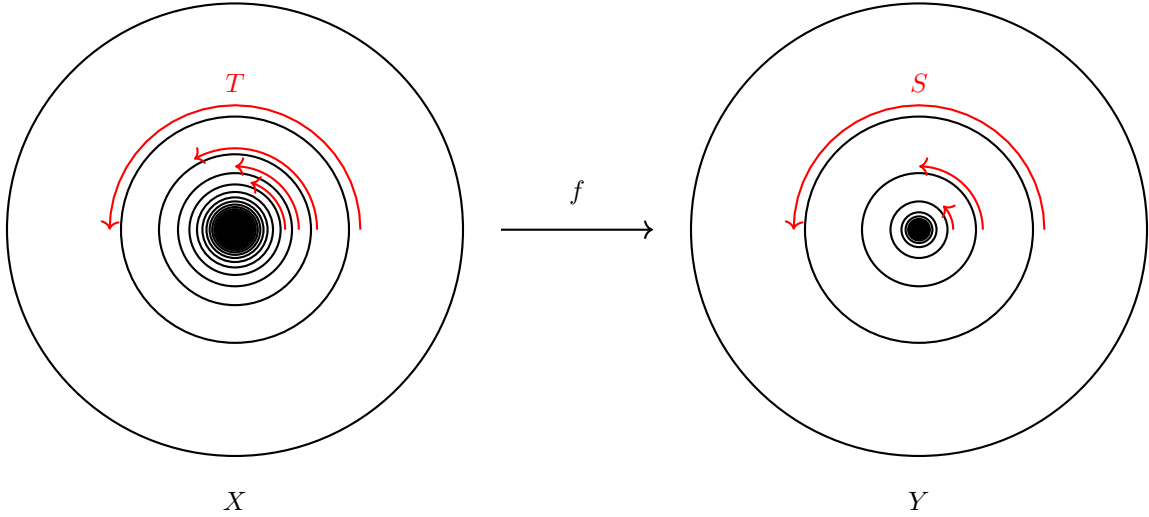
\begin{figure}[!ht]
    \centering
    \begin{tikzpicture}
    \begin{scope}[shift = {(0,0)}, scale = 3]
        \foreach \ii in {1,...,20}{
            \draw[thick] (0,0) circle (1/\ii);
        }
        \draw[thick, fill = black] (0,0) circle (1/21);
        \node[red] at (0,0.65) {$T$};
        \draw [red,thick,domain=0:180, ->] plot ({0.55*cos(\x)},{0.55*sin(\x)});
        \draw [red,thick,domain=0:120, ->] plot ({0.36*cos(\x)},{0.36*sin(\x)});
        \draw [red,thick,domain=0:90, ->] plot ({0.28*cos(\x)},{0.28*sin(\x)});
        \draw [red,thick,domain=0:72, ->] plot ({0.22*cos(\x)},{0.22*sin(\x)});
        \node at (0,-1.2) {$X$};
    \end{scope}
    \begin{scope}[shift = {(9,0)}, scale = 3]
        \foreach \ii in {1,2,4,8,13,17}{
            \draw[thick] (0,0) circle (1/\ii);
        }
        \draw[thick, fill = black] (0,0) circle (1/21);
        \node[red] at (0,0.65) {$S$};
        \draw [red,thick,domain=0:180, ->] plot ({0.55*cos(\x)},{0.55*sin(\x)});
        \draw [red,thick,domain=0:90, ->] plot ({0.28*cos(\x)},{0.28*sin(\x)});
        \draw [red,thick,domain=0:45, ->] plot ({0.15*cos(\x)},{0.15*sin(\x)});
        \node at (0,-1.2) {$Y$};
    \end{scope}
    \draw[thick, ->] (3.5,0) to (5.5,0);
    \node at (4.5,0.5) {$f$};
    \end{tikzpicture}
    \caption{A representation of the actions $\ZZ\curvearrowright X$ and $\ZZ \curvearrowright Y$.}
    \label{fig:circulito}
\end{figure}

Now let $A \subset \NN\setminus \{0\}$ be some infinite set and let $Y\subset \CC$ be given by \[Y = \{0\} \cup \left\{ z \in \CC : |z| = \frac{1}{n} \mbox{ for some } n \in A \right\}.\] Similarly, the map $S\colon Y\to Y$ given by $S(z) = z\exp\left(2\pi i |z|\right)$ is a homeomorphism which induces an action $\ZZ \curvearrowright Y$.

Let $f \colon X \to Y$ be the map given by \[ f(z)   = \begin{cases} z & \mbox{ if } |z| = \frac{1}{n} \mbox{ for some } n \in A,  \\0 & \mbox{ otherwise. }  \end{cases}\]

It is clear from the definition that $f$ is a continuous, surjective map such that $f\circ T = S \circ T$ and thus a topological factor map from $\ZZ \curvearrowright X$ to $\ZZ \curvearrowright Y$. However, notice that $\mathtt{Per}(\ZZ\curvearrowright Y) = \ZZ A$, that is, the set of integers which are integer multiples of some element of $A$. 

Notice that as the set of prime numbers is decidable, then a subset $A$ of prime numbers is co-recursively enumerable if and only if $\ZZ A$ is co-recursively enumerable. In particular, if we take $A$ which is not co-recursively enumerable we would have that $\mathtt{Per}(\ZZ\curvearrowright Y)$ is not co-recursively enumerable either and thus $\ZZ \curvearrowright Y$ cannot be an EDS.
\end{proof}

\begin{proposition}\label{prop:factor-of-eds-which-is-not-eds}
	The class of EDS is not closed under topological factor maps, even if we restrict to zero-dimensional spaces.
\end{proposition}

\begin{proof}
    For each $n\in\N$, let $X_n=\{0,\dots,n\}$, and let $t_n = (1\ 2\ \dots\ n)$ be the permutation on $X_n$ which fixes $0$ and cyclically permutes $1\mapsto 2\mapsto 3 \mapsto\dots \mapsto n\mapsto 1$. Now define $Y_n$ as the product $X_0\times\dots\times X_n$, and $s_n$ as the map on $Y_n$ which applies $t_0,\dots, t_n$ component-wise. We define $\pi_{n+1}$ as the map $Y_{n+1}\to Y_n$ which removes the last component of the tuple, this is a factor map from $Y_{n+1}$ to $Y_n$ with their respective actions. 
    We now define $Y$ as the inverse limit of sequence
    \[ \dots\xrightarrow {\pi_{n+2}}\ Y_{n+1}\xrightarrow {\pi_{n+1}} Y_n \xrightarrow{\pi_{n} \ \ }  \dots \xrightarrow{\pi_1 \ \ } Y_0. \]
     
    It is straightforward to verify that the set $Y$ endowed with the component-wise action is a zero dimensional EDS. Let $A \subset \NN$. We now define a topological factor of $Y$ called $Y_A$. 
    
    For each $n$ define $g_n \colon X_n\to X_n$, as follows. For $n\not\in A$, $g_n$ sends everything to $0$. For $n\in A$, $g_n$ is the identity function on $X_n$.  
    Now define $f_n\colon Y_n\to Y_n$ as the function which on component $n$ applies $g_n$. Then $f_n$ is an endomorphism of $Y_n$ with the action induced by $s_n$. Finally, define $F$ as the endomorphism of $Y$ which  in component $n$ applies $f_n$ and let $Y_A$ be the image of $F$, which is a compact subset of $\prod_{n \in \NN} Y_n$ by the continuity of $F$. We endow $Y_A$ with the componentwise action inherited as a subsystem of $Y$.  
    
    As in the previous construction, it suffices to take $A$ as a set of primes which is not co-recursively enumerable. In this case, $Y_A$ is a zero-dimensional topological factor of the zero dimensional EDS $Y$, but it cannot be an EDS because its set of periods is exactly $\ZZ A$ which is not co-recursively enumerable.
\end{proof}

Even if the class of EDS is not closed under topological factor maps, it is reasonable to assume that a factor of an EDS should still satisfy some sort of computability constraint. With the aim of understanding this class in the case of zero-dimensional spaces, we introduce a notion of weak EDS which will turn out to be stable under topological factor maps.

\begin{definition}\label{def:WEDS}
    An action $\Gamma \curvearrowright X$ with $X \subset A^{\NN}$ is a \define{weak effective dynamical system} (WEDS) if for every clopen partition $\mathcal{P}= (P_i)_{i=1}^n$ of $X$ the subshift \[ Y(\Gamma \curvearrowright X, \mathcal{P}) = \{ y \in \{1,\dots,n\}^{\Gamma} : \mbox{ there is } x \in X \mbox{ such that for every } g \in \Gamma,\ g^{-1}x \in P_{y(g)}\}.  \]
    is effective.
\end{definition}

Notice that for an action $\Gamma \curvearrowright X$ with $X\subset A^{\NN}$ to be a WEDS, it suffices to check the condition on the clopen partitions $\mathcal{P}_n = \{[w] : w \in A^n \mbox{ and } [w]\cap X \neq \varnothing\}$ as this family of clopen partitions refines any other clopen partition.  As $\Gamma \curvearrowright X$ can be obtained as the inverse limit of the shift action on the $Y(\Gamma \curvearrowright X, \mathcal{P}_n)$, we may also define a WEDS as an action that can be obtained as the inverse limit of a sequence of effective subshifts. Notice that we do not require the sequence to be uniform in this definition.

\begin{remark}
    Notice that any expansive WEDS $\Gamma \curvearrowright X$ is automatically an EDS, as in this case $\Gamma \curvearrowright X$ is topologically conjugate to $Y(\Gamma\curvearrowright X,\mathcal P)$ for any partition with small enough diameter and thus we have that $\Gamma \curvearrowright X$ is an EDS.
\end{remark}
\begin{proposition}\label{prop:0d-EDS_is_WEDS}
    If $\Gamma\curvearrowright X$ is a zero dimensional EDS, then it is a WEDS.
\end{proposition}
\begin{proof}
    Let $\mathcal P$ be a cover of $X$ which consists of disjoint clopen sets and let $S$ be a finite set of generators for $\Gamma$. Let us consider the action of the free group $F(S)\curvearrowright X$ induced by the action $\Gamma \curvearrowright X$. We proved in \Cref{prop:effective-covers-give-effective-subshifts} that the subshift $Y(F(S)\curvearrowright X,\mathcal P)$ is effective and thus an EDS. We now prove that $Y(\Gamma \curvearrowright X,\mathcal P)$ is an EDS. For this it suffices to prove that every $w\in F(S)$ with $\underline{w}=1_\Gamma$ acts trivially on  $Y(\Gamma \curvearrowright X,\mathcal P)$, as then we have a topological conjugacy between $\Gamma\curvearrowright Y(\Gamma \curvearrowright X,\mathcal P)$ and $\Gamma\curvearrowright Y(F(S) \curvearrowright X,\mathcal P)$.

    Assume that $wy\neq y$ for some $y$ in $Y(F(S) \curvearrowright X,\mathcal P)$ in order to obtain a contradiction. We can assume that $wy$ and $y$ differ in $1_{F(S)}$ by shifting $y$. By definition, this means that there is some element $x\in X$ such that $x\in P_{y(1_{F(S)})}$ and $x\in P_{wy(1_{F(S)})}$. This is a contradiction, as the computable map associated to $w$ is the identity and thus the sets $P_{y(1_{F(S)})}$, $P_{wy(1_{F(S)})}$ are disjoint. 
\end{proof}

Next we show that the class of WEDS is closed under topological factors.

\begin{proposition}\label{prop:factor_of_WEDS_is_WEDS}
    Let $\Gamma\curvearrowright X$ be a WEDS and $\Gamma\curvearrowright Y$ be a zero-dimensional topological factor. Then $\Gamma\curvearrowright Y$ is a WEDS.
\end{proposition}

\begin{proof}
    Denote by $f \colon X \to Y$ the topological factor map and let $\mathcal{P} = (P_i)_{i=1}^n$ be a clopen partition of $Y$. Then $\mathcal{Q} = (f^{-1}(P_i))_{i=1}^n$ is a clopen partition of $X$. As $X$ is WEDS, we obtain that $Y(\Gamma\curvearrowright X, \mathcal{Q})$ is an effective subshift. As $Y(\Gamma\curvearrowright X, \mathcal{Q}) = Y(\Gamma\curvearrowright Y, \mathcal{P})$, we obtain that $Y$ is a WEDS.
\end{proof}

As an immediate corollary of~\Cref{prop:0d-EDS_is_WEDS} and~\Cref{prop:factor_of_WEDS_is_WEDS} we obtain:

\begin{corollary}
    Every zero-dimensional topological factor of a zero-dimensional EDS is a WEDS.
\end{corollary}

Finally, putting this result together with~\Cref{thm:zero_dim_effective_extension} we obtain the following general result.

\begin{proposition}\label{prop:factor-of-EDS-is-WEDS}
    Let $\Gamma$ be a finitely generated and recursively presented group, $\Gamma \curvearrowright X$ be an EDS and $\Gamma\curvearrowright Y$ be a zero-dimensional topological factor. Then $\Gamma\curvearrowright Y$ is a WEDS.
\end{proposition}

\begin{proof}
    As $\Gamma$ is recursively presented, by~\Cref{thm:zero_dim_effective_extension} there is a zero-dimensional EDS extension of $\Gamma \curvearrowright X$. Then $\Gamma\curvearrowright Y$ is a topological factor of this zero-dimensional extension and thus by~\Cref{prop:0d-EDS_is_WEDS} we obtain that it is a WEDS.
\end{proof}

We have proved that for zero dimensional systems, the class of EDS is different to the class of WEDS. Indeed,  the example constructed in~\Cref{prop:factor-of-eds-which-is-not-eds} is not an EDS, but it follows from~\Cref{prop:factor-of-EDS-is-WEDS} that is a WEDS.

\begin{corollary}
    Let $\Gamma \curvearrowright X$ be an EDS and $\Gamma \curvearrowright Y$ be an expansive and zero-dimensional topological factor. If $X$ is zero-dimensional or $\Gamma$ is recursively presented then $\Gamma \curvearrowright Y$ is an EDS.
\end{corollary}

\begin{corollary}\label{cor:EDSsubshifts_closed_under_top_factors}
    The class of effective subshifts is closed under topological factor maps.
\end{corollary}

We remark that~\Cref{cor:EDSsubshifts_closed_under_top_factors} can also be obtained through the Curtis-Hedlund-Lyndon theorem~\cite[Theorem 1.8.1]{ceccherini-SilbersteinC09}. We finish this section with the following question which we were unable to answer.

\begin{question}
    Is it true that any WEDS is the topological factor of some EDS?
\end{question}

		\Addresses
		
		\bibliographystyle{abbrv}
		\bibliography{ref}

\begin{thebibliography}{10}

\bibitem{ABS2017}
N.~Aubrun, S.~Barbieri, and M.~Sablik.
\newblock A notion of effectiveness for subshifts on finitely generated groups.
\newblock {\em Theoretical Computer Science}, 661:35--55, 2017.

\bibitem{AubrunSablik2010}
N.~Aubrun and M.~Sablik.
\newblock Simulation of effective subshifts by two-dimensional subshifts of
  finite type.
\newblock {\em Acta Applicandae Mathematicae}, 126:35--63, 2013.

\bibitem{Barbieri_2019_DA}
S.~Barbieri.
\newblock A geometric simulation theorem on direct products of finitely
  generated groups.
\newblock {\em Discrete Analysis}, 2019.

\bibitem{BS2018}
S.~Barbieri and M.~Sablik.
\newblock A generalization of the simulation theorem for semidirect products.
\newblock {\em Ergodic Theory and Dynamical Systems}, 39(12):3185--3206, 2019.

\bibitem{Barbieri_Sablik_Salo_2021}
S.~Barbieri, M.~Sablik, and V.~Salo.
\newblock Groups with self-simulable zero-dimensional dynamics.
\newblock {\em arXiv:2104.05141}, 2021.

\bibitem{barbieri2023soficity}
S.~Barbieri, M.~Sablik, and V.~Salo.
\newblock Soficity of free extensions of effective subshifts.
\newblock {\em arXiv:2309.02620}, 2023.

\bibitem{Berger1966}
R.~Berger.
\newblock {\em The Undecidability of the Domino Problem}.
\newblock American Mathematical Society, 1966.

\bibitem{Birman_Brendle_2005_braidgroupssurvey}
J.~S. Birman and T.~E. Brendle.
\newblock Chapter 2 - braids: A survey.
\newblock In W.~Menasco and M.~Thistlethwaite, editors, {\em Handbook of Knot
  Theory}, pages 19--103. Elsevier Science, Amsterdam, 2005.

\bibitem{Bowen1978}
R.~Bowen.
\newblock On axiom {A} diffeomorphisms.
\newblock {\em American Mathematical Society (CBMS Regional Conference Series
  in Mathematics, 35), Providence, RI}, 1978.

\bibitem{Brattka2021}
V.~Brattka and P.~Hertling.
\newblock {\em Handbook of computability and complexity in analysis}.
\newblock Springer Nature, Cham, Switzerland, June 2021.

\bibitem{CannonFloydParry_thompsongroups}
J.~W. Cannon, W.~J. Floyd, and W.~R. Parry.
\newblock Introductory notes on {R}ichard {T}hompson's groups.
\newblock {\em Enseign. Math. (2)}, 42(3-4):215--256, 1996.

\bibitem{Carrasco_nicanor_subgroup_membership_2023}
N.~Carrasco-Vargas.
\newblock The geometric subgroup membership problem.
\newblock {\em arXiv:2303.14820}, 2023.

\bibitem{ceccherini-SilbersteinC09}
T.~Ceccherini-Silberstein and M.~Coornaert.
\newblock {\em Cellular Automata and Groups.}
\newblock Springer, 2009.

\bibitem{Cenzer2008}
D.~Cenzer, S.~A. Dashti, and J.~L.~F. King.
\newblock Computable symbolic dynamics.
\newblock {\em {MLQ}}, 54(5):460--469, 2008.

\bibitem{Chung2014}
N.-P. Chung and H.~Li.
\newblock Homoclinic groups, {IE} groups, and expansive algebraic actions.
\newblock {\em Inventiones Mathematicae}, 199(3):805--858, May 2014.

\bibitem{Coornaert2006-du}
M.~Coornaert and A.~Papadopoulos.
\newblock {\em Symbolic dynamics and hyperbolic groups}.
\newblock Lecture Notes in Mathematics. Springer, Berlin, Germany, 1993
  edition, Nov. 2006.

\bibitem{DurandRomashchenkoShen2010}
B.~Durand, A.~Romashchenko, and A.~Shen.
\newblock Effective closed subshifts in 1d can be implemented in 2d.
\newblock In {\em Fields of Logic and Computation}, pages 208--226. Springer
  Nature, 2010.

\bibitem{Hadamard1898}
J.~Hadamard.
\newblock Les surfaces {\`a} courbures oppos{\'e}es et leurs lignes
  g{\'e}od{\'e}siques.
\newblock {\em Journal de Math{\'e}matiques Pures et Appliqu{\'e}es}, 4:27--74,
  1898.

\bibitem{Hanf1974}
W.~{Hanf}.
\newblock Nonrecursive tilings of the plane. i.
\newblock {\em The Journal of Symbolic Logic}, 39(2):283--285, 1974.

\bibitem{Hedlund1969}
G.~A. Hedlund.
\newblock Endomorphisms and automorphisms of the shift dynamical system.
\newblock {\em Mathematical systems theory}, 3(4):320--375, 1969.

\bibitem{MorseHedlund1938}
G.~A. Hedlund and M.~Morse.
\newblock Symbolic dynamics.
\newblock {\em American Journal of Mathematics}, 60(4):815--866, 1938.

\bibitem{Hochman2009b}
M.~Hochman.
\newblock On the dynamics and recursive properties of multidimensional symbolic
  systems.
\newblock {\em Inventiones Mathematicae}, 176(1):131--167, 2009.

\bibitem{HochmanMeyerovitch2010}
M.~Hochman and T.~Meyerovitch.
\newblock A characterization of the entropies of multidimensional shifts of
  finite type.
\newblock {\em Annals of Mathematics}, 171(3):2011--2038, 2010.

\bibitem{KerrLiBook2016}
D.~Kerr and H.~Li.
\newblock {\em Ergodic Theory}.
\newblock Springer International Publishing, 2016.

\bibitem{LindSchmidt99}
D.~Lind and K.~Schmidt.
\newblock Homoclinic points of algebraic {$\mathbb{Z}^d$}-actions.
\newblock {\em Journal of the American Mathematical Society}, 12(4):953--980,
  1999.

\bibitem{Lind1990}
D.~Lind, K.~Schmidt, and T.~Ward.
\newblock Mahler measure and entropy for commuting automorphisms of compact
  groups.
\newblock {\em Inventiones Mathematicae}, 101(1):593--629, Dec. 1990.

\bibitem{morris_1977}
S.~A. Morris.
\newblock {\em Pontryagin Duality and the Structure of Locally Compact Abelian
  Groups}.
\newblock London Mathematical Society Lecture Note Series. Cambridge University
  Press, 1977.

\bibitem{Myers1974}
D.~{Myers}.
\newblock Nonrecursive tilings of the plane. ii.
\newblock {\em The Journal of Symbolic Logic}, 39(2):286--294, 1974.

\bibitem{Rett2013}
R.~Rettinger and K.~Weihrauch.
\newblock Products of effective topological spaces and a uniformly computable
  {T}ychonoff theorem.
\newblock {\em Logical Methods in Computer Science}, 2013.

\bibitem{Robinson1971}
R.~M. Robinson.
\newblock Undecidability and nonperiodicity for tilings of the plane.
\newblock {\em Inventiones Mathematicae}, 12:177--209, 1971.

\bibitem{Rogers:1987:TRF:28907}
H.~Rogers~Jr.
\newblock {\em Theory of Recursive Functions and Effective Computability}.
\newblock MIT Press, Cambridge, MA, USA, 1987.

\bibitem{Schmidt1995}
K.~Schmidt.
\newblock {\em Dynamical Systems of Algebraic Origin}.
\newblock Springer Nature, 1995.

\bibitem{sipser2012introduction}
M.~Sipser.
\newblock {\em Introduction to the theory of computation}.
\newblock Wadsworth Publishing, Belmont, CA, 3 edition, June 2012.

\bibitem{Turing1936}
A.~M. Turing.
\newblock On computable numbers, with an application to the
  {Entscheidungsproblem}.
\newblock {\em Proceedings of the London Mathematical Society. Second Series},
  42:230--265, 1936.

\bibitem{Wang1961}
H.~Wang.
\newblock Proving theorems by pattern recognition, {II}.
\newblock In {\em Computation, Logic, Philosophy}, pages 159--192. Springer
  Netherlands, 1961.

\end{thebibliography}
	\end{document}